\documentclass[11pt,reqno]{amsart}

\usepackage[a4paper,left=35mm,right=35mm,top=30mm,bottom=30mm,marginpar=25mm]{geometry}

\usepackage{amsmath}
\usepackage{amssymb}
\usepackage{amsthm}
\usepackage[utf8]{inputenc}
\usepackage{eurosym}
\usepackage{graphicx}
\usepackage{hyperref}
\usepackage{dsfont}
\usepackage{dsfont}
\usepackage{xcolor,colortbl}
\usepackage{comment}

\usepackage{bm}
\usepackage{caption}
\usepackage{subcaption}
\usepackage[linesnumbered,ruled]{algorithm2e}

\allowdisplaybreaks

\usepackage{ifthen}

\makeindex

\numberwithin{equation}{section}

\newtheorem{theorem}{Theorem}
\numberwithin{theorem}{section}
\newtheorem{lemma}[theorem]{Lemma}

\newtheorem{proposition}[theorem]{Proposition}

\newtheorem{definition}[theorem]{Definition}
\newtheorem{hyp}{Assumption}
\newtheorem{problem}{Problem}

\newtheorem{remark}[theorem]{Remark}

\DeclareMathOperator{\Lip}{Lip}
\DeclareMathOperator{\arginf}{arginf}
\DeclareMathOperator{\supp}{supp}

\author{Yuri Ashrafyan}\thanks{King Abdullah University of Science and Technology (KAUST), CEMSE Division, Thuwal 23955-6900.  Saudi Arabia. e-mail yuri.ashrafyan@kaust.edu.sa}
\author{Diogo Gomes}\thanks{King Abdullah University of Science and Technology (KAUST), CEMSE Division, Thuwal 23955-6900.  Saudi Arabia. 
e-mail: diogo.gomes@kaust.edu.sa}


\thanks{
      The authors were partially supported by King Abdullah University of Science and Technology (KAUST) baseline funds and KAUST OSR-CRG2021-4674.
}

\begin{document}

\title[A Fully-discrete Semi-Lagrangian scheme for a price formation MFG]{A Fully-discrete Semi-Lagrangian scheme \\ for a price formation MFG model}

\date{\today}

\begin{abstract}
Here, we examine a fully-discrete Semi-Lagrangian scheme for a 
mean-field game price formation model. 
We show the existence of the solution of the discretized problem and that it is monotone as a multivalued operator.
Moreover, we show that the limit of the discretization converges to the weak solution of the continuous price formation mean-field game using monotonicity methods. 
Numerical simulations demonstrate that this scheme can provide results efficiently, comparing favorably with other methods in the examples we tested.
\end{abstract}

\maketitle

\section{Introduction}

This paper investigates first-order price formation mean-field games (MFGs), modeling interactions within large populations of agents. 
Mean-field game theory, pioneered by \cite{ll1,ll2} and \cite{Caines1, Caines2}, extends classical optimal control problems to scenarios involving large populations of interacting agents.  The price formation problem draws on both optimal control principles \cite{gomes2018mean}
 and mean-field game techniques to model how individual decisions and market-wide dynamics influence each other.
Price formation mean-field games provide a framework for analyzing market behavior. By modeling the interactions of many agents, these systems may offer insights into resource allocation and market dynamics. Here, we develop a numerical method for approximating solutions of price formation mean-field games using Semi-Lagrangian (SL) schemes.

More precisely, we consider the following price formation problem.
\begin{problem}\label{pro1}
	Given initial distribution $\bar{m} \in \mathcal{P}(\mathbb{R})$,
	where $\mathcal{P}(\mathbb{R})$ is the space of probability measures on $\mathbb{R}$,
	terminal cost $\bar{u} \in C^2(\mathbb{R})$, and supply $Q \in C^2([0,T])$, and a uniformly convex Hamiltonian, $H \in C^1(\mathbb{R}^2)$.
	Find functions $u, m : \mathbb{R} \times [0,T] \rightarrow \mathbb{R}$ with $m \geq 0$, and $\varpi : [0,T] \rightarrow \mathbb{R}$, satisfying
	\begin{equation}\label{eq:MFG}
		\begin{cases}
			-u_t(x,t) + H(x, \varpi(t) + u_x(x,t) )=0, &\quad (x,t)\in \mathbb{R} \times [0,T]
			\\
			m_t(x,t) - \big(D_pH(x, \varpi(t) + u_x(x,t))m(x,t)\big)_x =0, &\quad (x,t)\in \mathbb{R} \times [0,T],
			\\
			-\int_{\mathbb{R}} D_pH(x, \varpi(t) + u_x(x,t))m(x,t) dx = Q(t), &\quad t\in [0,T],
		\end{cases}
	\end{equation}
	subject to initial-terminal conditions
	\begin{equation}\label{eq:initial_terminal}
		\begin{cases}
			u(x,T)=\bar{u}(x), \qquad x\in \mathbb{R},
			\\
			m(x,0)=\bar{m}(x), \qquad x\in \mathbb{R},
		\end{cases}
	\end{equation}
\end{problem}

In Problem~\ref{pro1}, $D_pH$ is a partial derivative of $H$ with respect to the second variable. 
The unknown is the triplet $(u, m, \varpi)$, where the value function, $u$, solves the first equation in the viscosity sense, and the probability distribution of the agents, $m$, solves the second equation in the distributional sense.  Finally, 
the price, $\varpi$, is determined by the balance condition, the third equation in \eqref{eq:MFG}, which guarantees that supply $Q(t)$ meets demand exactly.
A typical agent's state, represented by $x\in \mathbb{R}$, denotes their assets. 
At time $t$, the distribution of assets is encoded in the probability measure $m(\cdot, t)$.
Every agent trades assets at a price given by the time-dependent function $\varpi(t)$.
Section~\ref{sec:background} briefly recalls the preceding problem's derivation and motivation.

Under the assumptions discussed in Section~\ref{sec:Assumptions}, in \cite{gomes2018mean}, the authors proved the existence and uniqueness of a solution $(u,m,\varpi)$ to Problem~\ref{pro1}.   
This model assumes deterministic supply, a simplification that aids analysis but limits applicability in scenarios with uncertain supply.  Nonetheless, deterministic models offer a foundation for understanding more complex markets and can
 be directly applicable when accurate forecasts are available. 

The complexity of these models often renders analytical solutions infeasible.  Numerical methods allow for analyzing richer market dynamics, providing insights into previously unsolvable problems, and informing decision-making in practical applications.
A significant challenge is the forward-backward nature of traditional MFGs, compounded by the additional balance condition.
Researchers have suggested various numerical techniques for solving standard MFG systems; for example,
finite difference schemes and Newton-based approaches were presented in \cite{CDY} and \cite{DY}, respectively, for standard MFGs without balance conditions.
For the specific model considered here, in \cite{ashrafyan2022variational}, 
the authors introduced a potential function that transforms the price formation MFG into a convex variational problem. Then the problem was solved numerically by minimizing the variational problem discretized with finite differences. 

Price problems with a stochastic supply and price were introduced in \cite{Gomes2021price_noise} and \cite{gomes2023random_supply}.
Substantial effort was made to address the issue of common noise with machine learning (ML) techniques.
In \cite{Gomes:2023ML_price} proposes a ML approach employing dual recurrent neural networks and adversarial training to address the price formation MFG model with common noise.
An ML training process using a min-max characterization of control and price variables was developed in \cite{Gomes:2023AMO}.
Using the potential transformation, \cite{ashrafyan2022potential} solved the price problem with common noise using a recurrent neural network.  
These methods are suitable for common noise problems where standard numerical methods fail.
However, these methods can be slow for deterministic cases.  This was one of the original motivations for our work. 

Here, in the deterministic case, we 
develop efficient and accurate methods for deterministic price models.  We build on previous work on 
SL schemes for MFGs. 
Both semi-discrete and fully discrete SL schemes have been employed for first and second-order MFGs \cite{Camilli2012semi_discrete}, \cite{Carlini2014sl}, and \cite{Carlini:2013}.
In the Hamilton-Jacobi equation, we use a fully discrete SL scheme, and
for the transport equation, we use the adjoint of the linearization of the Hamilton-Jacobi equation.
For the balance condition, we use a simple numerical quadrature.
This discretization and the algorithm to solve it are detailed in Section~\ref{sec:Discretization}.
Properties and estimates related to the discretization are presented in Section~\ref{sec:prop_est}.
In Section~\ref{existence}, the existence of solutions to the discretized system is proved using Kakutani’s fixed point theorem. 
The concept of monotonicity for the discretized system is examined in Section \ref{monotonicity}. 
Moreover, it is shown that the limit of the discretization converges to a weak solution of Problem~\ref{pro1}.

Our approach addresses several gaps in existing numerical methods for price formation MFG: challenges with computational efficiency, accuracy in nonlinear scenarios, and suitability for large-scale validation of machine learning results. 
Our SL-based scheme offers advancements in all of these areas: it provides improved computational speed, robust handling of complex market dynamics, and the precision necessary for benchmarking ML models. 

These advantages are crucial for iterative analysis, policy testing, real-time applications, and the development of reliable, data-driven market insights.
A detailed comparison between our fully discrete SL, potential minimization, and ML methods appears in Section~\ref{sec:numerical_results}.

We gratefully acknowledge the referees' thorough reviews and constructive feedback, which substantially improved the quality and clarity of this manuscript.

\section{Background}\label{sec:background}

This section establishes the foundation for our price formation model.
  The model examines how prices form within a market where numerous agents interact, each trading motivated by profit maximization (here formulated
as cost-minimization).
Employing optimal control theory, we  analyze
how individual agents strategize to minimize costs associated
with asset holding, market impact, storage costs, and risk preferences.
Additionally, using mean-field game theory, we track the overall distribution of agents, impacting supply and demand.  These individual choices and their aggregate distribution, together with a market-clearing mechanism
 form the core mechanism that ultimately drives price dynamics within this model.

Agents have a state $X(t)\in \mathbb{R}$, corresponding to their assets. 
Let $\mathcal{A}_t$ be the set of bounded, measurable real-valued functions on 
the interval $[t,T]$.
An admissible control is a function $\alpha \in \mathcal{A}_t$.  Each agent
selects an admissible control that drives its trajectory
\begin{equation}\label{eq:dynamics}
	\begin{cases}
		dX(s)
		= \alpha(s) ds, \qquad s \in (t,T), \\
		X(t) = x,
	\end{cases}
\end{equation}
where $X(t)$ represents the agent's state at the time $t$.
The total cost for an agent is the sum of the terminal cost, $\bar{u}(x)$, and the integral of the running cost, 
\[
L(x,t,\alpha) = l_0(\alpha) + V(x) + \varpi(t)\alpha(t).
\]
Here $l_0$, the market impact term, 
represents costs like market impact or other trading expenses.
The preference potential, 
 $V(x)$, encodes the agent's asset level preference and can account for risk aversion, short-selling constraints, or storage costs.
 Lastly, $\varpi(t) \alpha(t)$ is the price agents pay by trading $\alpha$ units at a price $\varpi$. 
The agent aims to minimize the following cost functional:
\begin{equation}\label{eq:functional}
	J(x,t,\alpha) = \int_t^T L(X(s), s, \alpha(s))  ds +	\bar u(X(T)),
\end{equation}
by selecting appropriate admissible controls. 

The value function, $u(x,t)$,  is the infimum of the
 cost functional $J(x,t,\alpha)$
over all $\alpha \in \mathcal{A}_t$
\begin{equation}\label{eq:u_inf}
	u(x,t) := \inf_{\alpha \in \mathcal{A}_t} J(x,t,\alpha).
\end{equation}
The associated Hamiltonian, $H$,   is the Legendre-Fenchel transform of  $l_0 + V$, 
given by
\begin{equation}\label{eq:Legendre_transform}
	H(x,p)=\sup_{\alpha \in\mathbb{R}} \left\{- p \alpha - l_0(\alpha) - V(x) \right\}.
\end{equation}

Viscosity solutions offer a framework to handle potential non-differentiability that arises in such problems  \cite{BardiCapuzzo}.
In fact, $u$ is a viscosity solution to the Hamilton-Jacobi equation 
\[
\begin{cases}
	-u_t(x,t) + H(x, \varpi(t) + u_x(x,t))=0,
	\\
	u(x,T)=\bar{u}(x).
\end{cases}
\]

The dynamic programming principle for \eqref{eq:u_inf} provides 
a representation formula for the value function
\begin{equation}\label{eq:dpp}
	u(x,t-h) = \inf_{\alpha \in \mathcal{A}_t} \int_{t-h}^t  L(X(t-h), t-h, \alpha(s))  ds + u(X(t),t),
\end{equation}
where $X(t)$ solves \eqref{eq:dynamics}. Furthermore, the optimal control $\alpha^*$ is given in feedback form by
\begin{equation}\label{eq:feedback}
	\alpha^*(x,t) = -D_p H(x, \varpi(t) + u_x(x,t) ),
\end{equation} 
at all points where $u$ is differentiable.

The transport equation is the adjoint to the linearized Hamilton-Jacobi
equation.  This transport 
equation tracks the evolution of the distribution of the states over time
and incorporates the initial distribution, $\bar{m}(x)$, as an initial condition:
\[
\begin{cases}
	m_t(x,t) - \big(D_pH(x, \varpi(t) + u_x(x,t))m(x,t)\big)_x =0,
	\\
	m(x,0)=\bar{m}(x).
\end{cases}
\]

More precisely, given $(x, t) \in \mathbb{R}\times [0,T]$, consider the trajectory
\begin{equation*}
	\begin{cases}
		d\Phi_{x, t}(s) =  \alpha^*(\Phi_{x, t}(s),s)ds, \quad  \forall s\in(t,T) \\
		\Phi_{x, t}(t) = x.
	\end{cases}
\end{equation*}

Recall that push forward for a probability measure $\mu$ and a map $\psi$,  $\psi \# \mu$, is the probability measure defined by
\[
\int_\mathbb{R} \varphi d (\psi \# \mu) = \int_\mathbb{R} \varphi \circ \psi d \mu,
\]
for all bounded and continuous functions $\varphi$.

We obtain the density function $m$ from the initial measure $\bar{m}$ by pushing
forward with the flow $\Phi_{\cdot,0}(t)$:
\[
m(\cdot,t)= \Phi_{\cdot,0}(t) \#\bar{m}(\cdot).
\]
Note that for every $0 \leq t \leq t+h \leq T$, we have
\[
m(\cdot,t+h)= \Phi_{\cdot,t}(t+h)\# [\Phi_{\cdot,0}(t) \#\bar{m}(\cdot)] = \Phi_{\cdot,t}(t+h)\#  m(\cdot,t); 
\]
that is, 
\begin{equation}\label{eq:push_forward}
	\int_\mathbb{R} \varphi(x) m(x,t+h) dx = \int_\mathbb{R} \varphi (\Phi_{x,t}(t+h)) m(x,t) dx.
\end{equation}

Our model supposes a predefined supply function, denoted by $Q(t)$. 
The model imposes the equality between supply and demand:
\begin{equation}\label{eq:balance_alpha}
	\int_{\mathbb{R}} \alpha^*(x,t)m(x,t) dx = Q(t).
\end{equation}
This condition, expressed using the Hamiltonian, becomes
\begin{equation}\label{eq:balance_DpH}
	-\int_{\mathbb{R}} D_pH(x, \varpi(t) + u_x(x,t))m(x,t) dx = Q(t).
\end{equation}

\section{Main assumptions}\label{sec:Assumptions}

Throughout this paper, we work under the following assumptions as in \cite{gomes2018mean} and ensure the model is well-posed.

\begin{hyp}\label{hyp:H_l_0}
	Hamiltonian, $H$, is the Legendre-Fenchel transform of a uniformly convex Lagrangian, 
	\begin{equation}\label{eq:Legendre_transform}
		H(x,p)=\sup_{\alpha \in\mathbb{R}} \left\{- p \alpha - l_0(\alpha) - V(x) \right\},
	\end{equation}
	where  the market impact term, $l_0 \in C^2(\mathbb{R})$, is uniformly convex with a convexity constant $\kappa$; that is,  $l_0^{''}(\alpha)\geq \kappa >0$, and the preference potential $V$ is bounded from below.
\end{hyp}

The assumption that the market impact $l_0(\alpha)$ is independent of the agent's current state $x$ aligns well with price models, where these costs are typically distinct from preference costs $V(x)$. 
The convexity of $l_0$, which penalizes oscillating trading strategies, is fundamental to the stability and uniqueness of the MFG system's solutions.
Without it, the Hamiltonian becomes singular as in \eqref{eq:Legendre_transform} the supremum could be $+\infty$ if $p \neq 0$.

\begin{hyp}\label{hyp:H_bound}
	There exist constants $\theta>0$ and $C>0$, such that
	\[
	D^2_{pp} H(x,p) > \theta, \qquad |D^3_{ppp} H(x,p)| < C,
	\]
	for all $(x, p) \in \mathbb{R}^2$.
\end{hyp}

The uniform convexity of $l_0$ yields the strict convexity of $H$ with respect to the second variable and gives an upper bound for $D_{pp}^2 H$; that is, 
Assumption \ref{hyp:H_bound} implies $|D^2_{pp} H(x,p)| < C$.
The boundedness of the Hamiltonian's derivatives is essential for the existence of solutions.

\begin{hyp}\label{hyp:V_uT_m0_2nd_derivative}
	The potential $V$, the terminal cost $\bar{u}$, and the initial density $\bar{m}$ belong to $C^2(\mathbb{R})$.
	Furthermore, there exists a constant $C>0$ such that
	\[
	\|V\|_{C^2} \leq C, \quad \|\bar{u}\|_{C^2} \leq C, \quad |\bar{m}_{xx}| \leq C.
	\]
\end{hyp}

\begin{hyp}\label{hyp:V_uT_Lipschitz_convex}
	$V$ and $\bar u$ are globally Lipschitz and convex. 
\end{hyp}

The convexity of $V$ and $\bar u$ give regularity and uniqueness of solutions, see \cite{gomes2018mean}.

\begin{hyp}\label{hyp:Q}
	The supply function, $Q$, belongs to $C^1([0,T])$. 
\end{hyp}

\begin{hyp}\label{hyp:compactness}
	The initial distribution function, $\bar{m}$, has a mass of 1 and is compactly supported in $[-\bar{R}, \bar{R}] \subset \mathbb{R}$, for some positive number $\bar{R}\in \mathbb{R}$. 
\end{hyp}

The compactness assumption on initial distribution corresponds to the case when the agent's assets are bounded for the whole population, and ensures compact support of the density function $m(x,t)$ up to the terminal time $T$.

Under assumptions \ref{hyp:H_l_0}-\ref{hyp:compactness}, from \cite{gomes2018mean}, we have existence and uniqueness of solution $(u,m,\varpi)$, where $m$ is bounded, $u$ is Lipschitz, semiconcave, and differentiable in $x$.
Moreover, from \cite{ashrafyan2021duality}, $\varpi$ is Lipschitz continuous.

\section{Discretization}\label{sec:Discretization}


For hyperbolic conservation laws and Hamilton-Jacobi equations, SL schemes are known for their stability and ability to handle sharp gradients or nonsmooth solutions (see, e.g., \cite{BardiCapuzzo, FALCONE2002559, FalconeFerretti:1998, bonaventura_second_2021, calzola_semi-lagrangian_2023}).
In this section, we construct a discretized version of Problem~\ref{pro1}.  We use the dynamic programming principle and a fully discrete semi-Lagrangian scheme to discretize the value function $u$; for the density function m, we use trajectories of the dynamics \eqref{eq:dynamics}; finally, we discretize the balance condition to get an update rule for the price function $\varpi$.
The section concludes with an algorithm
to solve the discretized system.

To construct the discretization, we introduce the following notations and definitions.
For given positive numbers $\rho$ and $h$, define a space-time grid
\[
\mathcal{G}_{\rho,h}=\{(x_i, t_k)=(i \rho, k h) \, : \, i\in\mathbb{Z}, k=0,1,\ldots,N \},
\] 
such that $t_N = N h = T$, for a finite time $T$.
Consider the set of $\mathbb{P}_1$ basis functions $\{\beta_i\}_{i\in\mathbb{Z}}$, defined by 
\begin{equation}\label{eq:beta_i}
\beta_i(x) = \max\left\{1- \frac{|x-x_i|}{\rho}, \, 0\right\},
\end{equation}
which are compactly supported in $[x_i-\rho, x_i+\rho]$, $0\leq \beta_i(x)\leq1$ for all $x \in \mathbb{R}$, and $\beta_i(x_j)=\delta_{ij}$, where $\delta_{ij}$ is the Kronecker delta, and
\begin{equation}\label{eq:bi_sum}
    \sum_{i\in \mathbb{Z}} \beta_i(x) = 1,
\end{equation}
for every point  $x \in \mathbb{R}$. 
Given values $f(x_i)$ at the nodes $x_i$, we define the interpolation function as
\begin{equation}\label{eq:interpolation}
	I[f](x)=\sum_{i\in\mathbb{Z}} f(x_i) \beta_i(x).
\end{equation}
This interpolation function enables the Semi-Lagrangian discretization scheme to approximate values outside the grid $\mathcal{G}_{\rho,h}$.
A well-known estimate (see, e.g., \cite{Ciarlet, quartesaccosaleri07}) states  that for every bounded Lipschitz function $f$,  the following holds:
\begin{equation}\label{eq:Inter_Lipschitz}
    \sup_{x\in\mathbb{R}} \left| I[f](x) - f(x) \right| = O(\rho),
\end{equation}
and for $f \in C^2(\mathbb{R})$ with bounded first and second derivatives:
\begin{equation}\label{eq:Inter_smooth}
\sup_{x\in\mathbb{R}} \left| I[f](x) - f(x) \right| = O(\rho^2).
\end{equation}

We use the following notation. For a function $f(x,t)$ and indices $i\in\mathbb{Z}$, $k=0,1,\ldots,N$, $f_{i,k}$ denotes the value at node $(x_i,t_k)$ and $f_k$ denotes the vector ${f_{i,k}}_{i\in\mathbb{Z}}$. For a time-dependent function $g(t)$, $g_k$ denotes $g(t_k)$. For any function $f$, $f^q$ denotes its approximation at iteration $q$.


\subsection{Hamilton-Jacobi equation}

We approximate $u$ by a function $u^q$ by replacing the right-hand side of \eqref{eq:dpp} 
with the following Semi-Lagrangian discretization
\[
	u^q(x_i,t-h) = \inf_{\alpha \in\mathbb{R}} \left\{ I[u^q(\cdot, t)](y_i(\alpha)) + h L(x_i, t-h, \alpha)\right\} ,
\]
where $y_i(\alpha)=x_i + h \alpha$ is the discretization of $X(t)$.

For $i\in\mathbb{Z}$ and for $ k=N-1, N-2, \ldots, 0$ the discrete Hamilton-Jacobi equation is
\begin{equation}\label{eq:HJ_discretization}
	\begin{cases}
		\displaystyle
		u^q_{i,k} = \inf_{\alpha \in\mathbb{R}} \left\{ I[u_{k+1}^q](y_i(\alpha)) + h L_{i,k}(\alpha) \right\},
		\\
		u^q_{i,N} =
		\bar{u}_i.
	\end{cases}
\end{equation}

At the point $(x_i, t_k)$, let $\Lambda_{i,k}^q$ be the set of controls that minimize the discrete cost functional in \eqref{eq:HJ_discretization}:
\begin{equation}\label{eq:Lambda}
\Lambda_{i,k}^q = \left\{\alpha^* \in \mathbb{R} : \;  
\alpha^* =\arginf_{\alpha \in\mathbb{R}} \left\{ I[u_{k+1}^q](y_i(\alpha)) + h L_{i,k}(\alpha)
		\right\} \right\}.
\end{equation}
By Lemma~\ref{lem:unique_minimizer}, the set $ \Lambda_{i,k}^q $ is uniformly bounded. The set contains finitely many minimizers due to two properties: the coercivity and convexity of $L_{i,k}(\alpha)$ restrict minimizers to a bounded region, and the piecewise linear structure of $I[u_{k+1}](y(x_i, \alpha))$ which consists of finitely many linear segments with possible non-smoothness at grid points.
We select one minimizer $ \alpha^{*q}_{i,k}$ from $\Lambda_{i,k}^q$ as the result of a numerical minimization algorithm. An alternative would be to choose, for example, the leftmost element:
\begin{equation}\label{eq:alpha_star_i_k}
    \alpha^{*q}_{i,k} := \min \left\{ \Lambda_{i,k}^q \right\}.
\end{equation}
The minimizer is typically unique except under special conditions (eg symmetry). The general case of multiple minimizers is addressed in Section \ref{existence}, where we discuss the existence of solutions using probability measures on the set of minimizers, and in Section~\ref{monotonicity}, where we interpret the discretized system as a multivalued monotone operator.
The transport equation depends on both $u$ and the chosen $\alpha^*$.

\subsection{Continuity equation}

The discretization of the Hamilton-Jacobi equation offers a method to approximate the value function. 
Now, we consider the continuity equation, which complements this, approximating the time evolution of density functions.

Substituting $\varphi(x)$ by  $\beta_i(x)$
in equation \eqref{eq:push_forward} , we get
\[
\int_\mathbb{R} \beta_i(x) m(x,t+h) dx = \int_\mathbb{R} \beta_i (\Phi_{x,t}(t+h)) m(x,t) dx.
\]
Taking into account that  $\beta_i(x_j) = \delta_{ij}$, 
the preceding expression suggests 
the discretization formula for $m(x,t)$ 
\begin{equation}\label{eq:FP_discretization}
	\begin{cases}
		m^q_{i,k+1} = \sum_{j\in\mathbb{Z}} \beta_i(\Phi_{x_j,t_k}(t_{k+1})) m^q_{j,k}, \\
		m^q_{i,0} = \bar{m}_i,
	\end{cases}
\end{equation}
for $i\in\mathbb{Z}$ and $ k=1,2, \ldots, N$, where we use the discretized flow given by
\[
\Phi_{x_i,t_k}(t_{k+1}) = x_i + h \alpha^{*q}_{i,k}.
\]

\subsection{Balance condition}
Now, we address the discretization of the last term in \eqref{eq:MFG}, the balance condition:  
\[
-\int_{\mathbb{R}} D_pH(x, \varpi(t) + u_x(x,t))m(x,t) dx = Q(t).
\]
Notice that since Hamiltonian is separable, $D_p H$ is a function of $p$ alone, and from the feedback formula \eqref{eq:feedback}, we have
\[
u_x = D_pH^{-1}( -\alpha^*) - \varpi.
\]
We can substitute this representation into the balance condition and discretize it. This gives the following implicit update rule for the price function at iteration $q$:
\begin{equation}\label{eq:balance_discretization}
\sum_{i\in\mathbb{Z}} D_pH(x, \varpi^{q+1}_k - \varpi^q_k + D_pH^{-1}(- \alpha_{i,k}^{*q}) ) m^q_{i,k} \Delta x = -Q_k,
\end{equation}
for $k=0,1,\ldots,N-1$.

According to the prior discussion, the discretization 
of the MFG system \eqref{eq:MFG}-\eqref{eq:initial_terminal} 
is represented by the system of equations \eqref{eq:HJ_discretization}, \eqref{eq:FP_discretization}, and \eqref{eq:balance_discretization}. 
We employ an iterative algorithm outlined below to solve this discretized system efficiently.

\subsection{Algorithm}

\begin{enumerate}
	\item[0)]
	Initialize $q=0$, select an initial guess for the price, $\varpi^q$, and fix a tolerance parameter, $\varepsilon>0$. 

	\item[1)]  For $k=N-1, N-2, \ldots, 0$ and $i\in\mathbb{Z}$, compute the value function, $u$, using the discretization of the Hamilton-Jacobi equation:
	\begin{equation*}
		\begin{cases}
			\displaystyle u^q_{i,k} = \inf_{\alpha \in\mathbb{R}} \left\{  I[u^q_{k+1}](y_i(\alpha)) + h \, (l_0(\alpha)  + \varpi^q_k \alpha + V_i) \right\}, \\
			u^q_{i,N} = \bar{u}_i,
		\end{cases}
	\end{equation*}
    and find optimal controls:
	\[
	\alpha_{i,k}^{*q}= \arginf_{\alpha \in\mathbb{R}} \left\{ I[u^q_{k+1}](y_i(\alpha)) + h \, (l_0(\alpha)  + \varpi^q_k \alpha + V_i) \right\}.
	\]

	\item[2)] For $k=0,1,\ldots,N-1$ and $i\in\mathbb{Z}$, compute the distribution function, $m^q$, using the discretization of the Fokker-Planck equation:
	\begin{equation*}
		\begin{cases}
			\displaystyle m^q_{i,k+1}=\sum_{j\in\mathbb{Z}} \beta_i \left(y_j(\alpha_{j,k}^{*q}) \right) m^q_{j,k},\\
			m^q_{i,0} = \bar{m}_i.
		\end{cases}
	\end{equation*}

	\item[3)]  For $k=0,1,\ldots,N-1$, update the price, $q \to q+1$, from the following implicit equation
	\[
	\sum_{i\in\mathbb{Z}} D_pH(x_i, \varpi^{q+1}_k - \varpi^q_k + D_pH^{-1}(- \alpha_{i,k}^{*q}) ) m^q_{i,k} \Delta x = -Q_k.
	\]

	\item[4)] If
	\[
	\| \varpi^{q+1} - \varpi^q \|_{\infty} < \varepsilon,
	\]
      stop, otherwise, set $\varpi^q := \varpi^{q+1} $, and repeat steps $1)-4)$.
\end{enumerate} 
Note that the price function is defined only up to the $N-1$ time step, as there is no optimization at the terminal time.

To avoid notational clutter, the subsequent sections omit the explicit dependence on $q$ for symbols representing approximation functions when the context is clear.

\section{Some properties and estimates}\label{sec:prop_est}

This section establishes the theoretical foundation for our discretization scheme.  First, in Lemmas \ref{lem:disc_prop} and \ref{lem:disc_estimates}, we prove several properties crucial to its well-definedness, consistency, and computational behavior.  Then,
in Lemma \ref{lem:unique_minimizer}, we demonstrate the scheme's existence and boundedness of minimizers $\alpha^*$.  Finally, in Lemma \ref{lem:compactness}, we prove that compactly supported initial distributions remain compactly supported.

Let the right-hand side of the value function discretization in \eqref{eq:HJ_discretization} be denoted by
\[
S_{\rho, h}(f,i,k) := \inf_{\alpha
	\in\mathbb{R}} \left\{ I[f\lvert_{t=t_{k+1}}](y_i(\alpha)) + h L_{i,k}( \alpha) \right\}.
\]
  We examine the properties of this discretization in the subsequent lemmas, establishing conditions for the existence of a minimizer and providing explicit bounds on the minimizers.

\begin{lemma}\label{lem:disc_prop}	
Under Assumptions \ref{hyp:H_l_0} and \ref{hyp:V_uT_Lipschitz_convex}, 	$S_{\rho, h}(f,i,k)$ enjoys the following properties:
	\begin{enumerate}
		\item[i.] [Well-defined] 
		For $f$ bounded below function, there exists $\alpha \in \mathbb{R}$, at which the right-hand side attains its infimum.

		\item[ii.] [Monotone] For a fixed price $\varpi$, the discretization is monotone, i.e. if $u \leq v$, we have 
		\[
		S_{\rho, h}(u,i,k) \leq S_{\rho, h}(v,i,k).
		\]

		\item[iii.] [Consistent] Let $(\rho_n, h_n) \to 0$, $\rho_n^2/ h_n \to 0$, as $n \to \infty$ and grid points $(x^n_i,t^n_k)$ converge to $(x,t)$, then for every $f \in C^2(\mathbb{R} \times [0,T])$, 
		\[
		\lim_{n \to \infty} \frac{1}{h_n} \left[ f^n_{i,  k} - S_{\rho_n, h_n}(f,i, k) \right] = 
		-f_t(x,t) + H(x, \varpi + f_x).
		\]

		\item[iv.][Translation Invariance] For all $K \in \mathbb{R}$
		\[
		S_{\rho, h}(f+K,i,k) = S_{\rho, h}(f,i,k) + K.
		\]
	\end{enumerate}
\end{lemma}

\begin{proof}
	\begin{enumerate}
		\item[i.] 		
        From the definition of linear interpolation \eqref{eq:interpolation}, and property of basis functions $\beta_i$, it follows
        \[
            I[f\lvert_{t=t_{k+1}}](y_i(\alpha_n)) \geq \inf \{f_{k+1}\}.
        \]
  Let $\{\alpha_n^*\}$ be a minimizing sequence. 
        By the coercivity of $L$, we can restrict $\alpha_n^*$ to a compact set, then
		\[
		S_{\rho, h}(f,i,k) \geq \inf \{f_{k+1}\} +  h L_{i,k}( \alpha_n^*) \geq C,
		\]
		since $L$ is uniformly convex in $\alpha$, by Assumption \ref{hyp:H_l_0}. 
		\item[ii.] 
		Let $\alpha^*_{i,k}$ and $\tilde{\alpha}^*_{i,k}$ be minimizers for $S_{\rho, h}(u,i,k)$ and $S_{\rho, h}(v,i,k)$, respectively.
Because $u \leq v$, then 
        \[
        I[u_{k+1}](y_i(\tilde{\alpha}^*_{i,k})) \leq I[v_{k+1}](y_i(\tilde{\alpha}^*_{i,k})),
        \]
        and we get
		\begin{align*}
			S_{\rho, h}(u,i,k) & \leq I[u_{k+1}](y_i(\tilde{\alpha}^*_{i,k})) + h L_{i,k}( \tilde{\alpha}^*_{i,k})  \\
			& \leq I[v_{k+1}](y_i(\tilde{\alpha}^*_{i,k})) + h L_{i,k}( \tilde{\alpha}^*_{i,k})  = S_{\rho, h}(v,i,k).
		\end{align*}
		\item[iii.] 
		Recall that $H$ is the Legendre-Fenchel transform of $l_0 + V$ given by \eqref{eq:Legendre_transform},
		\begin{align*}
			& \lim_{n \to \infty} \frac{1}{h_n} \left[ f^n_{i,  k} - S_{\rho_n, h_n}(f,i, k) \right] \\
			= & \lim_{n \to \infty} - \inf_{\alpha
				\in\mathbb{R}} \left\{ \frac{I[f\lvert_{t=t^n_{k+1}}](x^n_i + h_n \alpha) -  f^n_{i,  k} }{h_n} +  L(x^n_i, t^n_k, \alpha) \right\} \\
            = & \lim_{n \to \infty} - \inf_{\alpha
			\in\mathbb{R}} \left\{ \frac{f(x^n_i+h_n \alpha, t^n_{k+1}) -  f^n_{i,  k} }{h_n} + O\left(\frac{\rho^2_n}{h_n} \right) + L(x^n_i, t^n_k, \alpha) \right\} \\
			= &  - \inf_{\alpha
				\in\mathbb{R}} \left\{ f_t(x,t) + f_x(x,t) \alpha +  L(x, t, \alpha) \right\} \\
			= & - f_t(x,t) + \sup_{\alpha
				\in\mathbb{R}} \left\{ - \alpha (f_x(x,t) + \varpi(t)) -  l_0(\alpha) - V(x) \right\} \\
			= & - f_t(x,t) + H(x, \varpi(t) + f_x(x,t)).
		\end{align*}
        On the third line, we used the following approximation for the $C^2$ function $f$, from the estimate of the interpolation function \eqref{eq:Inter_smooth}.
    \[
    I[f^n_{k+1}]\left( x^n_i + h_n \alpha \right) = f(x^n_i+h_n \alpha, t^n_{k+1}) + O\left(\rho^2_n \right).
    \]
Note that in the above identities, we can switch the limit with the infimum because
 $f \in C^2(\mathbb{R} \times [0,T])$ and by the coercivity of $L$, we can restrict $\alpha$
 to a compact set. 

		\item[iv.] $f\mapsto I[f]$ is linear and for any constant $K \in \mathbb{R}$, $I[K]=K$.\qedhere
	\end{enumerate}
\end{proof}

The next lemma shows that the distribution function is a probability measure and gives uniform estimates on value function and price, which are only dependent on the data of the problem but not on a grid.
\begin{lemma}\label{lem:disc_estimates}
	Suppose Assumptions \ref{hyp:H_l_0}--\ref{hyp:compactness} hold. 
	Let $(m, u, \varpi)$  be the solution to the discretized system \eqref{eq:HJ_discretization}, \eqref{eq:FP_discretization}, and \eqref{eq:balance_discretization}.
	Then the triplet $(m, u, \varpi)$ satisfies the following properties:
	\begin{enumerate}
		\item[i.] $\sum_{i\in\mathbb{Z}} m_{i,k} \Delta x = 1$, for all $k=0,\ldots, N$, and $m_{i,k} \geq 0$, for all $i, k$,

		\

		\item[ii.] { $|u_{k}| \leq C$ uniformly over compact sets, and $u_k$ are globally Lipschitz, with $\Lip(u_k) \leq  \Lip(\bar{u}) + T \Lip\left( V \right)$, for all $k=0,\ldots, N$,}

		\

		\item[iii.] $\| \varpi \|_{\infty} \leq C$,
	\end{enumerate}
\end{lemma}

\begin{proof}
	\begin{enumerate}
		\item[i.] 
        The positivity of $m_{i,k}$ follows directly from the 
         definition \eqref{eq:FP_discretization}, the positivity of $\bar{m}$ and $\beta_i$ functions.		
        And conservation of the mass follows from discretization \eqref{eq:FP_discretization} and property of $\beta_i$ functions \eqref{eq:bi_sum}:
		\begin{align*}
			\sum_{i\in\mathbb{Z}} m_{i,k+1} \Delta x &= \sum_{i\in\mathbb{Z}} \sum_{j\in\mathbb{Z}} \beta_i \left( \Phi_{x_j,t_k}(t_{k+1}) \right) m_{j,k}  \Delta x \\
			&= \sum_{j\in\mathbb{Z}} m_{j,k}  \Delta x \sum_{i\in\mathbb{Z}} \beta_i \left( \Phi_{x_j,t_k}(t_{k+1}) \right) = \sum_{j\in\mathbb{Z}} m_{j,k}  \Delta x.
		\end{align*}
		To finish the proof, note that for $k=0$
		\[
		\sum_{j\in\mathbb{Z}} m_{j,0}  \Delta x = \sum_{j\in\mathbb{Z}} \bar{m}_{j}  \Delta x = 1,
		\]
    by Assumption \ref{hyp:compactness}.
      
		\

		\item[ii.] From the discretization for the value function \eqref{eq:HJ_discretization}, we can get an upper bound by taking $\alpha^* = 0$:
		\[
		\begin{aligned}
			u_{i, N-1} & \leq \bar{u}_i + h l_0(0) + h V_i, \\ 
			u_{i, N-2} & \leq u_{i,N-1} + h l_0(0) + h V_i \\
			& \leq \bar{u}_i +  2h l_0(0) + 2h V_i, \\ 
			& \ldots \\
			u_{i, 0} & \leq \bar{u}_i + T l_0(0) + T V_i.
		\end{aligned}
		\]
  
        Similarly, the lower bound follows from $(i.)$ of Lemma~\ref{lem:disc_prop} and Assumption~\ref{hyp:V_uT_Lipschitz_convex}. 
        
		The Lipschitz estimate is obtained by using the discretization for the value function \eqref{eq:HJ_discretization}, and Assumption \ref{hyp:V_uT_Lipschitz_convex}.
		At the point $x_j$, we have
		\begin{align*}
			u_{j,N-1} \leq & I[ \bar{u}]\left( y_j(\alpha_{i,N-1}^*) \right) + h \, \varpi_{N-1} \alpha_{i,N-1}^* + h l_0(\alpha_{i,N-1}^*) + h V_j \\
			& + I[ \bar{u}]\left( y_i( \alpha_{i,N-1}^*) \right) - I[ \bar{u}]\left( y_i( \alpha_{i,N-1}^*) \right) + h V_i - h V_i  \\
			\leq & u_{i,N-1}  + \Lip \left( \bar{u} \right) \left| x_i - x_{j} \right| + h \, \Lip \left( V \right) \left| x_i - x_{j} \right|.
		\end{align*}
		Because the preceding estimate is symmetric with respect to $i$ and $j$, we get
		\[
		\Lip\left( u_{N-1} \right) \leq \Lip\left( \bar{u} \right) + h \, \Lip\left( V \right).
		\]
		Applying iteratively the same procedure for $k=2, 3, \ldots, N$, yields:
		\begin{align*}
			\Lip\left( u_{N-k} \right) 
			& \leq \Lip\left( u_{N-k+1} \right) + h \, \Lip\left( V \right) 
			\leq \Lip\left( \bar{u} \right) + k h \, \Lip\left( V \right) \\
			&\leq \Lip\left( \bar{u} \right) + T \, \Lip\left( V \right).
		\end{align*}

		\item[iii.] This estimate follows from combining the results of $i.$ and $ii.$, and the discretization for   the balance condition in \eqref{eq:MFG}. 

  Define the functional
  \[
   J(\varpi) := \sum_{i\in\mathbb{Z}} D_pH(\varpi + U_i) m_{i,k} \Delta x,
  \]
  with $U_i$ bounded.
  Then by Assumption \ref{hyp:H_bound}, it follows that
    \[
   J'(\varpi) > \theta > 0.
  \]
  This implies $J(\varpi)$ is a monotone strictly increasing function. We replace the  balance condition in \eqref{eq:MFG}
 by 
  \[
  J(\varpi_k) = -Q_k.
  \]
  If $Q$ bounded, we conclude that $\| \varpi \|_{\infty} \leq C$.
 \qedhere
	\end{enumerate}
\end{proof}
Next, we prove the boundedness of the minimizer in 
 \eqref{eq:HJ_discretization}.
\begin{lemma}\label{lem:unique_minimizer}
Suppose 
Assumptions \ref{hyp:H_l_0}--\ref{hyp:compactness} hold.
Then, there exists a positive constant $C$, such that 
			\[
			|\alpha^*_{i,k}|< C
			\]
			for all $i,k$.
\end{lemma}
\begin{proof}
        Define a function
		\[
		g(\alpha) := I[u_{k+1}](y_i(\alpha)) + h L_{i,k}( \alpha).
		\]
		Because $\alpha^*_{i,k}$ is a minimizer of $g(\alpha)$, we have
  		\begin{align*}
		0 & \leq  g(0) - g(\alpha^*_{i,k})  =  u_{i,k+1} - I[u_{k+1}](x_i + h \alpha^*_{i,k}) + h L_{i,k}(0) - h L_{i,k}(\alpha^*_{i,k}) \\
& \leq C h \, |\alpha^*_{i,k}| + h L_{i,k}(0) - h L_{i,k}(\alpha^*_{i,k}),
		\end{align*}
where we used the Lipschitz continuity of $u_{k+1}$, by Lemma~\ref{lem:disc_estimates}. 
The preceding inequality gives the estimate:
  \[
 \frac{ L_{i,k}(\alpha^*_{i,k}) }{|\alpha^*_{i,k}|} \leq C + \frac{ L_{i,k}(0) } {|\alpha^*_{i,k}|}.
  \]
 Because $L$ is coercive, we have 
  \[
 \lim_{|\alpha| \rightarrow \infty} \frac{ L_{i,k}(\alpha) }{|\alpha|} =+\infty.
  \]
  This implies that $\alpha^*_{i,k}$ is bounded. 
\end{proof}

With the established bounds on the minimizer, $\alpha^*$ and Lipschitz continuity of $u_k$ with respect to space discretization, we are ready to prove the Lipschitz continuity of the value function in time, uniformly in $h$.
\begin{lemma}\label{lem:time_lipschitz}
	Suppose Assumptions \ref{hyp:H_l_0}--\ref{hyp:V_uT_Lipschitz_convex} hold. 
 Then there exists a positive constants $C$, such that
    \[
    |u_{i,k} - u_{i,l}| \leq C |t_k - t_l|,
    \]
{ uniformly in $h$,  over compact sets. }
\end{lemma}

\begin{proof}
   In the definition of the value function discretization \eqref{eq:HJ_discretization}, taking the minimizer $\alpha^*_{i,k} = 0$, we obtain the following inequality:
     \[
		\begin{aligned}
			u_{i,k} &\leq  u_{i,k+1} +  h (l_0(0) + V_i) \\
            & =  u_{i,k+1} + ((k+1)h - k h) (l_0(0) + V_i) \\
            & =  u_{i,k+1} + (t_{k+1} - t_k) (l_0(0) + V_i), 
		\end{aligned}
		\]
 which gives us an estimate:
  \begin{equation}\label{eq:u_k_k1}
  u_{i,k} - u_{i,k+1} \leq C |t_{k+1} - t_k|.
  \end{equation}

Next using the
boundedness of the optimal control $\alpha^*$ in Lemma \ref{lem:unique_minimizer}, we get
    \[
	\begin{aligned}
        u_{i,k} & = I[u_{k+1}](x_i + h \, \alpha^*_{i,k}) + h l_0(\alpha^*_{i,k}) +h \varpi_k \alpha^*_{i,k}+h V_i\\
        & \geq I[u_{k+1}](x_i + h \, \alpha^*_{i,k}) + h C,
   	\end{aligned}
		\]
 where $C$ depends on $l_0$, lower bound of $V$, the uniform bound in $\varpi$, and the bound on $\alpha^*$.
 Now, consider the following difference:
    \begin{equation}\label{eq:u_k1_k}
		\begin{aligned}
			u_{i,k+1} - u_{i,k} & \leq  u_{i,k+1} - I[u_{k+1}](x_i +  h \, \alpha^*_{i,k}) - h C \\
            & \leq \Lip(u_{k+1}) \, | h \, \alpha^*_{i,k}|   + h \, |C|   \\
            & \leq \left(\Lip(\bar{u}) + T \Lip (V) \right)  C^* \,  h  + |C| \,   h \\
            & \leq C |t_{k+1} - t_k|,
		\end{aligned}
		\end{equation}
where, in the third line we use the space Lipschitz estimate $ii.$ of Lemma~\ref{lem:disc_estimates},  and $C^*$ is the bound on $\alpha^*_{i,k}$, using Lemma~\ref{lem:unique_minimizer}

From inequalities \eqref{eq:u_k_k1} and \eqref{eq:u_k1_k}, it follows:
\[
 |u_{i,k} - u_{i,k+1}| \leq C |t_{k+1} - t_k|.
\]
Then, for any two integers $0 \leq k , \, l \leq N $, we get
\[
 |u_{i,k} - u_{i,l}| \leq C |t_k- t_l|.
\]
\end{proof}

Finally, we show that under the compactness assumption on the initial distribution function, $\bar{m}$, there is a space and time discretization relation ensuring that the agent's assets remain bounded.
\begin{lemma}\label{lem:compactness}
Suppose Assumptions  \ref{hyp:H_l_0}-\ref{hyp:V_uT_m0_2nd_derivative}, and \ref{hyp:compactness} hold, for some $\bar{R}>0$, and 
\[
\frac{\rho}{h} < C,
\] 
for some positive constant $C$.
Then, $m_{k}$ is compactly supported, that is, there exists a positive constant $R \geq \bar{R}$, such that
$\supp (m_{k}) \subseteq [-R,R]$, for all $k=1, 2, \ldots, N$.
\end{lemma}
\begin{proof}
By the definition of $m_{k+1}$ in \eqref{eq:FP_discretization}, we have
\[
m_{i,1} = \sum_{j\in\mathbb{Z}} \beta_i(y_j(\alpha^*_{j,1})) \bar{m}_j.
\]
Recall that $\beta_i$ are compactly supported in $[x_i-\rho, x_i+\rho]$, the initial distribution function, $\bar{m}$, is compactly supported in $[-\bar{R}, \bar{R}]$, and from Lemma~\ref{lem:unique_minimizer}, $|\alpha^*_{j,k}| < C^*$, for some constant $C^*>0$.
These yield that for $m_{1}$ to have non-zero values the following relations must hold:
\[
\begin{cases}
& x_j - h C^* -\rho < x_i < x_j + h C^* + \rho , \\
& -\bar{R} < x_j < \bar{R}.
\end{cases}
\]
Which means that $m_1$ may have non-zero values only for $x_i \in [-R_1,R_1]$, where
\[
R_1 = \bar{R} + h C^* + \rho.
\]
Applying iteratively the same reasoning for $k=2, 3, \ldots, N$, gives that  $m_{k}$ is compactly supported in $[-R_k,R_k]$, where
\[
R_k = \bar{R} + k h C^* + k \rho < \bar{R} + T C^* + T \frac{\rho}{h} < \bar{R} + T C^* + T C .
\]
Setting $R = \bar{R} + T C^* + T C $ completes the proof.
\end{proof}

\section{Existence of solutions}
\label{existence}

In this section, we establish the existence of solutions to our discretized problem by using Kakutani's fixed point theorem.  
We recall that a multivalued map \( \Psi \) from a set \( X \) to a set \( Y \) is a map such that for each \( x \in X \), there exists a subset \( S \subseteq Y \), called the value of \( \Psi \) at \( x \), denoted \( \Psi(x) = S \). 
For Kakutani's theorem, we require that $X=Y$ and that
\( \Psi(x) \) is non-empty for all \( x \in X \), so \( \Psi : X \to 2^X \setminus \{\emptyset\} \), where \( 2^X \) denotes the set of all subsets of \( X \).

We begin by defining the set $X$. Then, we build 
a multivalued map
\[
\Psi: X\rightarrow 2^{X}\setminus \{\emptyset\},
\]
which has a closed graph, and for every $x\in X$ the set $\Psi (x)$ is non-empty and convex. 
As we will see, a fixed point of this map corresponds to a solution of the discrete MFG. 

Since the value functions $u_k$ are uniformly Lipschitz, as established in Lemma~\ref{lem:disc_estimates}, its subdifferentials of $I[u_k]$, $p_{i,k}$ are bounded. 
Particularly, we denote
\[
K_p = \{ p \in \mathbb{R} \mid |p| \leq L_u \},
\]
where $L_u$ is the Lipschitz constant of $u$.
Then, there exists a compact and convex set $K_\varpi$, such that for any set of  probability measures $B = \{ B_{i,k} \}$  supported on $K_p$, and density function $m$, the equation 
    \[
    \sum_{i \in \mathbb{Z}} m_{i,k} \Delta x \, \int_{K_p} D_p H(\varpi  + p) \,  d B_{i,k}(p) = - Q_k.
    \]
    has a unique solution $\varpi \in K_\varpi$, since $D_pH$ is strictly monotone continuous function.
We define also the compact set 
\[
K_\alpha := (l_0')^{-1}(-K_p - K_\varpi). 
\]  

The set $X$ is defined as the set of pairs
$(\varpi, B)$, where for each $(i,k)$, 
$\varpi^k\in K_\varpi$
and
$B_{ik}$ is a probability measure supported in $K_p$. This set is compact and convex. 

The mapping $\Psi$ is constructed through the following steps:

\paragraph{Step 1.}{\bf Density function}

For each probability measure $B_{ik}$ we build a probability measure $\mu_{ik}$ by
\[
\mu_{ik}=(l_0')^{-1}(-\varpi_k-p)\#B_{ik}(p).
\]
        For each given probability measure $\mu_{i,k}$ build the distribution function $m$ using the discretized transport equation:
    \[
		\begin{cases}
			\displaystyle m_{i,k+1} = \sum_{j \in \mathbb{Z}} m_{j,k} \, \int_{ K_\alpha } \beta_i(x_j + h \alpha) \,  d \mu_{j,k}(\alpha),\\
			m_{i,0} = \bar{m}_i.
		\end{cases}
	\]

    {\bf Continuity.}  
    The continuity here is straightforward, if $\varpi^n\to \varpi$ and $B^n\rightharpoonup B$, we have 
    $m^n_{i,k} \to m_{ik}$.

\paragraph{Step 2.}   {\bf Price function.} 
Construct the new price function $\tilde \varpi = \{ \tilde \varpi_k \}_{k=0}^{N-1} $ by solving the discretized balance condition for each time step $k$:
\[
    \sum_{i \in \mathbb{Z}} m_{i,k} \Delta x \, \int_{ K_\alpha  } D_p H(\tilde \varpi_k  - \varpi_k - l_0'(\alpha)) \,  d \mu_{i,k}(\alpha) = -Q_k.
\]
    From the definitions of $K_\varpi$ and $K_\alpha$, we have that $- \varpi_k - l_0'(\alpha) \in K_p$, and, thus the equation has a unique solution $\tilde \varpi_k \in K_\varpi$.

  {\bf Continuity.} 
The continuity here is also straightforward, if $m_{ik}^n\to m_{ik}$, $\mu_{ik}^n\rightharpoonup\mu_{ik}$ and $\varpi^n\to \varpi$, we see that $\tilde \varpi^n\to \tilde \varpi$.

\paragraph{Step 3.} {\bf Value Function.} 
        Now, we
        build the value function, $u$, using the discretization of the Hamilton-Jacobi equation:
    \[
		\begin{cases}
			\displaystyle u_{i,k} = \inf_{\alpha \in \mathbb{R}}  \left\{ I[u_{k+1}]\big(x_i + h \alpha\big) + h ( l_0(\alpha) + \tilde \varpi_k \alpha + V_i) \right\}, \\
			u_{i,N} = \bar{u}_i.
		\end{cases}
	\]
    Let 
    \[
g_{i,k}(\alpha) = I[u_{k+1}^n](x_i + h\alpha) + h(l_0(\alpha) +  \tilde\varpi_k \alpha + V_i).
\]
    Then, we obtain a finite set of all possible optimal controls $\alpha_{i,k}$:
\[
        \Lambda_{i,k} := \arginf_{\alpha \in K_\alpha} \left\{ g_{i,k}(\alpha)\right\}, 
\]
where   
\[
        g_{i,k}= I[u_{k+1}]\big(x_i + h \alpha\big) + h ( l_0(\alpha) + \tilde \varpi_k \alpha + V_i).
\]
    Notice that from the optimality condition for optimal controls $\alpha$, we have a relation:
    \[
    \alpha_{i,k}  = (l_0')^{-1} (- p_{i,k} - \tilde\varpi_{k}).
    \] 
    where $p_{i,k}$ is a subdifferential of $I[u_{k+1}]$.

{\bf Continuity.} 

Suppose that $ \tilde\varpi^n\to  \tilde\varpi$. Consider the sequence of functions 
\[
g_{i,k}^n(\alpha) = I[u_{k+1}^n](x_i + h\alpha) + h(l_0(\alpha) +  \tilde\varpi_k^n \alpha + V_i).
\]
The value functions $u_{i,k}^n$ are recursively defined by
\[
u_{i,k}^n = \inf_{\alpha \in K_\alpha} g_{i,k}^n(\alpha),
\]
with terminal condition $u_{i,N}^n = \bar{u}_i$.
We aim to prove:
\begin{enumerate}
\item The sequence $u_{i,k}^n$ converges to $u_{i,k}$ for all $i$ and $k$
\item The limit function satisfies $u_{i,k} = \inf_{\alpha \in K_\alpha} g_{i,k}(\alpha)$
\item The optimal controls $\alpha_{i,k}^n \in \operatorname*{arginf}_{\alpha \in K_\alpha} g_{i,k}^n(\alpha)$ converge to $\alpha_{i,k} \in \operatorname*{arginf}_{\alpha \in K_\alpha} g_{i,k}(\alpha)$
\end{enumerate}

At the terminal time step we have
\[
u_{i,N}^n = \bar{u}_i,
\]
which is constant for all $n$. We proceed by induction, 
assuming that for some $k+1 \leq N$
\[
u_{i,k+1}^n \to u_{i,k+1} \quad \text{as } n \to \infty.
\]
We need to show that $u_{i,k}^n \to u_{i,k}$ as $n \to \infty$, where
\[
u_{i,k} = \inf_{\alpha \in K_\alpha} g_{i,k}(\alpha). 
\]

By the inductive hypothesis
we have
\[
g_{i,k}^n(\alpha) \to g_{i,k}(\alpha)
\]
uniformly over $K_\alpha$ as $n \to \infty$.
Since $K_\alpha$ is compact and $g_{i,k}^n(\alpha)$ converges uniformly, the infimum converges:
\[
u_{i,k}^n = \inf_{\alpha \in K_\alpha} g_{i,k}^n(\alpha) \to \inf_{\alpha \in K_\alpha} g_{i,k}(\alpha) = u_{i,k}.
\]
Let $\alpha_{i,k}^n \in \operatorname*{arginf}_{\alpha \in K_\alpha} g_{i,k}^n(\alpha)$. By compactness of $K_\alpha$, there exists a convergent subsequence $\alpha_{i,k}^{n_j} \to \alpha_{i,k}$ as $j \to \infty$
and by uniform continuity
\[
\alpha_{i,k} \in \operatorname*{arginf}_{\alpha \in K_\alpha} g_{i,k}(\alpha).
\]
Accordingly, we have
\begin{enumerate}
\item $u_{i,k}^n \to u_{i,k}$ for all $i$ and $k$
\item $u_{i,k} = \inf_{\alpha \in K_\alpha} g_{i,k}(\alpha)$
\item $\alpha_{i,k}^n$ converge (along subsequences) to optimal controls $\alpha_{i,k}$
\end{enumerate}
Note that while multiple subsequences of $\{\alpha_{i,k}^n\}$ may converge to different limit points due to potential non-uniqueness of minimizers, all limit points are optimal controls for the limiting problem.

\paragraph{Step 4.}{\bf Relaxed controls}

For each $(i,k)$, consider the set $\Lambda_{ik}\subset K_{\alpha}$ of all optimal controls. 
Let $\hat \Lambda_{ik}=\tilde \varpi+l_0'(\Lambda_{ik})$. Consider the set $\mathcal{B}_{ik}$ of all
probability measures supported on $\hat \Lambda_{ik}$, and denote by $\mathcal{B}$ the cartesian product of these sets 
over all indices $(i,k)$. 

We define
\[
\Psi(\varpi, B)=\{(\tilde \varpi, \tilde B): \tilde B\in \mathcal{B}\}.
\]
Note that this set is compact and convex.

\paragraph{Existence of a fixed point}
We have constructed a multivalued mapp $\Psi$ from a non-empty, compact, and convex subset $X$ into itself. The mapping $\Psi$ satisfies the following properties:
\begin{itemize}
    \item For each $(\varpi,B) \in X$, the image $\Psi(\varpi)$ is a non-empty, convex and compact set. 
   
    \item The mapping $\Psi$ is upper hemicontinuous; that is, for any sequence $\{\varpi^n, B^n\} \subset X$ converging to $(\varpi,B)$, and any sequence $(\tilde \varpi^n,\tilde B^n )\in \Psi(\varpi^n)$ converging to $(\tilde\varpi, \tilde B)$, we have $(\tilde\varpi, \tilde B) \in \Psi((\varpi,B))$.
\end{itemize}

Thus, the conditions of Kakutani's fixed point theorem are satisfied.
Therefore, there exists at least one fixed point $\varpi^* \in K_\varpi^N$ such that $\varpi^* \in \Psi(\varpi^*)$. 
This fixed point corresponds to a solution of our discretized problem.

\section{Monotonicity}
\label{monotonicity}

In this section, we discuss the interpretation of the discretization \eqref{eq:HJ_discretization}, \eqref{eq:FP_discretization} and \eqref{eq:balance_discretization} as a multivalued map.
This allow us to show that the limit of the discretization is a weak solution to Problem~\ref{pro1}.

Set $\Omega_T = \mathbb{R} \times [0,T]$, and denote
\[
\begin{aligned}
& D = (C^\infty(\Omega_T) \cap C([0,T], \mathcal{P})) \times (C^\infty(\Omega_T) \cap W^{1, \infty}(\Omega_T)) \times C^\infty([0,T]), \\
& D_b = \{(m, u, \varpi) \in D \, | \, m(x, 0) = \bar{m}(x), u(x, T) = \bar{u}(x)\}\\
& D_+^b = \{(m, u, \varpi) \in D_b \, | \, m \geq 0, \, m(\cdot, t)\ \text{ compactly supported}  \}.  
\end{aligned}
\]
In the continuous case, under the convexity condition of the map $p \mapsto H(x,p)$, the operator
\[
\mathcal{A}
\begin{bmatrix}
	m \\
	u \\
	\varpi
\end{bmatrix}
=
\begin{bmatrix}
	u_t(x,t) - H(x, \varpi(t) + u_x(x,t) ) \\
	m_t(x,t) - \big(D_pH(x, \varpi(t) + u_x(x,t))m(x,t)\big)_x \\
	\int_{\mathbb{R}} D_pH(x, \varpi(t) + u_x(x,t))m(x,t) dx + Q(t)
\end{bmatrix}
\]
associated with Problem~\ref{pro1} is monotone; that is, it satisfies 
\begin{equation}\label{eq:monotonicity}
	\Big\langle
	\mathcal{A}\begin{bmatrix}
		\tilde{w}
	\end{bmatrix}
	-
	\mathcal{A}\begin{bmatrix}
		w
	\end{bmatrix}
	,
	\begin{bmatrix}
		\tilde{w}
	\end{bmatrix}
	-
	\begin{bmatrix}
		w
	\end{bmatrix}
	\Big\rangle
	\geq 0,
\end{equation}
where $ w = \{m, u, \varpi\}$ and  $ \tilde{w} = \{\tilde{m}, \tilde{u}, \tilde{\varpi}\}$ are arbitrary elements from $D_+^b$ (see \cite{gomes2018mean}, Proposition 9).

Here, we use the definition of a weak solution similar to one in \cite{FG2}, equation (1.6) (see, also \cite{FeGoTa20}).
\begin{definition}
A triplet $ w = \{m, u, \varpi\} \in \mathcal{P}(\Omega_T) \times L^\infty(\Omega_T) \times L^\infty([0,T])$,  with $m \geq 0$, and compactly supported, is a weak solution induced by monotonicity to Problem~\ref{pro1}, if it satisfies the variational inequality
\begin{equation}\label{eq:weak_solution}
	\Big\langle
	\mathcal{A}\begin{bmatrix}
		\tilde{w}
	\end{bmatrix}
	,
	\begin{bmatrix}
		\tilde{w}
	\end{bmatrix}
	-
	\begin{bmatrix}
		w
	\end{bmatrix}
	\Big\rangle
	\geq 0,
\end{equation}
for every $ \tilde{w}\in D_+^b$ .
\end{definition}

Note, that a classical solution to \eqref{eq:MFG} is always a weak solution. 
In the case of a regular $H$, and smooth $ \{m, u, \varpi\} \in D_+^b$, a weak solution is also a classical solution in the set $m>0$.
Indeed, take a triplet $\tilde{w} = \{\tilde{m}, \tilde{u}, \tilde{\varpi} \} \in C^\infty(\Omega_T) \times C^\infty(\Omega_T) \times C^\infty([0,T])$, and sufficiently small $\varepsilon > 0$ with $m + \varepsilon \tilde{m} \geq 0$, $\, \tilde{m}(x,0)=0, \, \tilde{u}(x,T) = 0 $.
In \eqref{eq:weak_solution} taking $w^\varepsilon = w + \varepsilon \tilde{w}$, and letting $\varepsilon \rightarrow 0$, gives 
\[
	\Big\langle
	\mathcal{A}\begin{bmatrix}
		w
	\end{bmatrix}
	,
	\begin{bmatrix}
		\tilde{w}
	\end{bmatrix}
	\Big\rangle
	\geq 0,
 \]
for arbtitrary $\tilde{w}$, which yields $A[w] = 0$.

\subsection{Multivalued map}

As before, let $\Lambda_{i,k}$ be the set of all controls that minimize the discrete cost functional in \eqref{eq:HJ_discretization} at the point $(x_i, t_k)$.
In general, $\Lambda_{i,k}$ may contain multiple values. Thus, as before, we consider  relaxed controls; that is,  probability measures $\mu_{i,k}$ supported on  $\Lambda_{i,k}$. 
While
the value $I[u_{k+1}](y_i(\alpha^*)) + h L_{i,k}( \alpha^*)$ is uniquely defined, the operator corresponding to the transport equation discretization depends on the particular choice of a relaxed control. 
Thus, it is natural to regard the discretization \eqref{eq:HJ_discretization}, \eqref{eq:FP_discretization} and \eqref{eq:balance_discretization} as a multivalued map that accounts for all possible choices of optimal controls.

Let $w = \{ m, u, \varpi \}$ be a discrete triplet. By  $\mu^*$ we denote a probability measure such that for each $(i,k)$, 
$\mu_{i,k}$ is supported on 
$\Lambda_{i,k}$. We denote by $\mathcal{M}[u]$ the set of all such probability measures.  
For any discrete triplet $w$  the operator $A_{\rho,h}$ is defined as the set of values given by 
\begin{equation}\label{eq:multivalued}
A_{\rho,h}
\begin{bmatrix}
	m \\
	u \\
	\varpi
\end{bmatrix}
=
\begin{bmatrix}
	-u_{i,k} + I[u_{k+1}](y_i(\alpha_{i,k}^{*})) + h L_{i,k}( \alpha_{i,k}^{*}) \\
	m_{i,k+1} - \sum_{n=0}^M \int \beta_i(y_n(\alpha)) m_{n,k} \, d\mu^*_{n,k}(\alpha)\\
	- h \left( \sum_{i\in\mathbb{Z}} \int \alpha m_{i,k} \, d\mu^*_{i,k}(\alpha) \Delta x - Q_k \right)
\end{bmatrix}, 
\end{equation}
where $\mu^* \in \mathcal{M}[u]$.

The third component of $A_{\rho,h}$ corresponds to the balance condition in 
terms of $\alpha^*$, as in \eqref{eq:balance_alpha}.  
The corresponding discretization, for $ k=0,1, \ldots, N-1$, becomes
\begin{equation}\label{eq:alpha*_balance_discretization}
	\sum_{i\in\mathbb{Z}} \int \alpha m_{i,k} \, d\mu^*_{i,k}(\alpha) \, \rho = Q_k.
\end{equation}
Note that the factor $h$ in the third line does not change the solutions, but it is needed to make the operator monotone.
The subsequent discussion proves that $A_{\rho,h}$ is a monotone operator. 

\begin{lemma}\label{lem:monotonicity}
	Under Assumption~\ref{hyp:H_l_0}, the operator $A_{\rho,h}$ in \eqref{eq:multivalued} is monotone. 
\end{lemma}
\begin{proof}
    Let $ w = \{m, u, \varpi\}$ and  $ \tilde{w} = \{\tilde{m}, \tilde{u}, \tilde{\varpi}\}$ be two given triplets with $m, \tilde{m} \geq 0$, and $\mu^{*} \in  \mathcal{M}[u] , \,  \tilde{\mu}^* \in \mathcal{M}[\tilde{u}]$.
Note that all $\alpha_{i,k}^{*} \in \Lambda_{i,k}$ are minimizers for the Hamilton-Jacobi discretization, and $\mu^*_{i,k}(\alpha)$ is a probability measure supported on that set, $\Lambda_{i,k}$, therefore, 
\[
I[u_{k+1}](y_i(\alpha_{i,k}^{*})) + h L_{i,k}( \alpha_{i,k}^{*}) = \int I[u_{k+1}](y_i(\alpha)) + h L_{i,k}( \alpha) \, d\mu^*_{i,k}(\alpha).
\]
The following relations between the Lagrangians $L$ and $\tilde{L}(x, t, \alpha) := l_0(\alpha) + V(x) + \tilde{\varpi}(t) \alpha$ hold true,
 \begin{equation}\label{eq:L_identity}
        L_{i,k}( \alpha) - \alpha \varpi_k = \tilde{L}_{i,k}( \alpha) - \alpha \tilde{\varpi}_k.
 \end{equation}
Further, we observe the following identity
\begin{equation}\label{eq:beta_I_identity}
	\begin{aligned}
		&\sum_{i\in\mathbb{Z}} \sum_{n\in\mathbb{Z}} \int \beta_{i} \left( y_n( \alpha) \right) m_{n,k} \, d\mu^*_{n,k}(\alpha)  v_{i,k+1} =
		\sum_{n\in\mathbb{Z}} m_{n,k} \sum_{i\in\mathbb{Z}} \int \beta_{i} \left( y_n( \alpha) \right) \, d\mu^*_{n,k}(\alpha) v_{i,k+1} \\
		=& \sum_{i\in\mathbb{Z}} m_{i,k} \int  \sum_{n\in\mathbb{Z}}  v_{n,k+1} \beta_{n} \left( y_i( \alpha) \right) \, d\mu^*_{i,k}(\alpha)  =
		\sum_{i\in\mathbb{Z}} m_{i,k} \int I \left[ v_{k+1} \right] \left( y_i( \alpha)  \right) \, d\mu^*_{i,k}(\alpha).
	\end{aligned}
\end{equation}

Now, we show that $A_{\rho,h}$ is monotone. 
\[
\begin{aligned}
&\Big\langle
A_{\rho,h}\begin{bmatrix}
\tilde{w}
\end{bmatrix}
-
A_{\rho,h}\begin{bmatrix}
w\end{bmatrix}
,
\begin{bmatrix}
\tilde{w}
\end{bmatrix}
-
\begin{bmatrix}
w\end{bmatrix}
\Big\rangle
\\
= & h \sum_{k=0}^{N-1} \rho \sum_{i\in\mathbb{Z}} \Big[
\Big( -u_{i,k} + \tilde{u}_{i,k}  +  \int I[u_{k+1}] \left( y_i(\alpha) \right) + h  L_{i,k}(\alpha) d\mu^*_{i,k}(\alpha) \\ 
& - \int I[\tilde{u}_{k+1}] \left( y_i(\alpha) \right)  +  h \tilde{L}_{i,k}(\alpha) d\tilde{\mu}^*_{i,k}(\alpha) \Big) (m_{i,k} - \tilde{m}_{i,k}) \\
&+ \Big( m_{i,k+1} - \tilde{m}_{i,k+1} - \sum_{n\in\mathbb{Z}} \int \beta_i(y_n(\alpha)) m_{n,k} \, d\mu^*_{n,k}(\alpha) \Big. \\
&+ \Big. \sum_{n\in\mathbb{Z}} \int \beta_i(y_n(\alpha)) \tilde{m}_{n,k} \, d\tilde{\mu}^*_{n,k}(\alpha) \Big) (u_{i,k+1} - \tilde{u}_{i,k+1}) \\
&+ h \Big( - \int \alpha m_{i,k} \, d\mu^*_{i,k}(\alpha) + \int \alpha \tilde{m}_{i,k} \, d\tilde{\mu}^*_{i,k}(\alpha) \Big) (\varpi_k - \tilde{\varpi}_k) \Big] 
\end{aligned}
\]
Using identities \eqref{eq:L_identity}, \eqref{eq:beta_I_identity}, and rearranging terms, we can simplify the expression:
\[
\begin{aligned}
&\Big\langle
A_{\rho,h}\begin{bmatrix}
\tilde{w}
\end{bmatrix}
-
A_{\rho,h}\begin{bmatrix}
w\end{bmatrix}
,
\begin{bmatrix}
\tilde{w}
\end{bmatrix}
-
\begin{bmatrix}
w\end{bmatrix}
\Big\rangle
\\
= & h \sum_{k=0}^{N-1} \rho \sum_{i\in\mathbb{Z}} \Big[
 m_{i,k} \Big( \int I[\tilde{u}_{k+1}] \left( y_i(\alpha) \right) + h \tilde{L}_{i,k}(\alpha) d\mu^*_{i,k}(\alpha) \Big. \\
& \Big.  - \int I[\tilde{u}_{k+1}] \left( y_i(\alpha) \right) + h  \tilde{L}_{i,k}(\alpha) d\tilde{\mu}^*_{i,k}(\alpha) \Big) \\
& + \tilde{m}_{i,k} \Big( \int I[u_{k+1}] \left( y_i(\alpha) \right)  + h  L_{i,k}(\alpha) d\tilde{\mu}^*_{i,k}(\alpha) \Big. \\
& -  \Big. \int I[u_{k+1}] \left( y_i(\alpha) \right)  + h  L_{i,k}(\alpha) d\mu^*_{i,k}(\alpha) \Big) \Big]. \\
\end{aligned}
\]
Noting that $\mu^*_{i,k}$ is supported on the minimizers of $I[u_{k+1}](y_i(\alpha)) + h L_{i,k}(\alpha)$ and $\tilde{\mu}^*_{i,k}$ is supported on the minimizers of $I[\tilde{u}_{k+1}](y_i(\alpha)) + h \tilde{L}_{i,k}(\alpha)$,  we have:
\[
 \int \left( I[\tilde{u}_{k+1}](y_i(\alpha)) + h \tilde{L}_{i,k}(\alpha) \right) d\tilde{\mu}^*_{i,k}(\alpha) 
 \leq 
 \int \left( I[\tilde{u}_{k+1}](y_i(\alpha)) + h \tilde{L}_{i,k}(\alpha) \right) d\mu^*_{i,k}(\alpha),
\]
and similarly,
\[
\int \left( I[u_{k+1}](y_i(\alpha)) + h L_{i,k}(\alpha) \right) d\mu^*_{i,k}(\alpha)
\leq 
\int \left( I[u_{k+1}](y_i(\alpha)) + h L_{i,k}(\alpha) \right) d\tilde{\mu}^*_{i,k}(\alpha) .
\]
Therefore, the entire expression is nonnegative:

\[
\Big\langle
A_{\rho,h}[w] - A_{\rho,h}[\tilde{w}]
,
w - \tilde{w}
\Big\rangle \geq 0.\qedhere
\]
\end{proof}

\begin{definition}
A triplet $ w = \{m, u, \varpi\}$, with  $m \geq 0$, and compactly supported, is a weak solution induced by monotonicity to the multivalued operator $A_{\rho, h}$ defined by \eqref{eq:multivalued}, if it satisfies the variational inequality
\begin{equation}\label{eq:discrete_weak_solution}
	\Big\langle
	A_{\rho, h}
    \begin{bmatrix}
		\tilde{w}
	\end{bmatrix}
	,
	\begin{bmatrix}
		\tilde{w}
	\end{bmatrix}
	-
	\begin{bmatrix}
		w
	\end{bmatrix}
	\Big\rangle
	\geq 0,
\end{equation}
for every $ \tilde{w} = \{ \tilde{m}, \tilde{u}, \tilde{\varpi} \}$, with $\tilde{m} > 0 $, $\tilde{m}_{i,0} = \bar{m}_i$, $\tilde{u}_{i,N} = \bar{u}_i$, and probability measures $\tilde{\mu}^*_{i,k}$ supported on $\tilde{\Lambda}_{i,k}$.
\end{definition}

It is worth noting that classical solutions are always weak solutions. 
On the continuous problem, due to lack of regularity, weak solutions may fail to be classical solutions.
But, here, in the discrete case, there is no issue with regularity, and a weak solution is a classical solution whenever $m>0$.
\begin{lemma}\label{lem:weak_strong_sol}
Let Assumption \ref{hyp:H_l_0} and \ref{hyp:compactness} hold.
Let $ w = \{m, u, \varpi\}$ be a weak solution for the operator $A_{\rho,h}$, then $A_{\rho,h} [w] = 0$, for some selection of $\mu \in  \mathcal{M}[u]$, at all grid points where $m>0$.
\end{lemma}

\begin{proof}
Take an $\varepsilon > 0$, and a triplet $\tilde{w} = \{\tilde{m}, \tilde{u}, \tilde{\varpi} \}$, $\tilde{m}_{i,0} = 0$ and $\tilde{u}_{i,N} = 0$.
We require $\tilde m$ to be supported on the set $m>0$
but do not impose $\tilde m>0$. 
Accordingly, for $\varepsilon$ small enough
$m+\varepsilon \tilde m\geq 0$. 
Let $w^\varepsilon = w + \varepsilon \tilde{w}$.
Let $\Lambda_{i,k}$ be the set of minimizers for $w$ and ${\Lambda}_{i,k}^\varepsilon$ the set of minimizes for $w^\varepsilon$. 
Note that any sequence $\alpha^\varepsilon_{i,k} \in \Lambda^\varepsilon_{i,k}$ will have a subsequence
converging to some 
$ \alpha^*_{i,k}\in \Lambda_{i,k}$, when $\varepsilon \rightarrow 0$.
Furthermore, 
any sequence $\mu^\varepsilon \in \mathcal{M}[u^\varepsilon]$
will have a weakly converging subsequence to some $\mu \in  \mathcal{M}[u]$.

Because $w$ is a weak solution, we have the following inequality:
\[
\begin{aligned}
0 \leq &
\Big\langle
A_{\rho,h}\begin{bmatrix}
w + \varepsilon \tilde{w}\end{bmatrix},
\begin{bmatrix}
\tilde{w}
\end{bmatrix}
\Big\rangle
\\
= & 
h \sum_{k=0}^{N-1} \rho \sum_{i\in\mathbb{Z}} \Big[
\Big( - u_{i,k} - \varepsilon \tilde{u}_{i,k} +  \int I[u_{k+1} + \varepsilon \tilde{u}_{k+1}] \left( y_i(\alpha) \right) \, d\mu^\varepsilon_{i,k}(\alpha) \\
&+ h \int \big( l_0(\alpha) + V_i + (\varpi_k + \varepsilon \tilde{\varpi}_k) \alpha \big) \, d\mu^\varepsilon_{i,k}(\alpha) \Big) \tilde{m}_{i,k} \\
&+ \Big( m_{i,k+1} + \varepsilon \tilde{m}_{i,k+1} - \sum_{n \in \mathbb{Z}} \int \beta_i(y_n(\alpha)) (m_{n,k} + \varepsilon \tilde{m}_{n,k}) \, d\mu^\varepsilon_{n,k}(\alpha) \Big) \tilde{u}_{i,k+1} \\
&+ h \Big( - \int \alpha (m_{i,k} + \varepsilon \tilde{m}_{i,k}) \, d\mu^\varepsilon_{i,k}(\alpha) \Big) \tilde{\varpi}_k \Big]. 
\end{aligned}
\]
Using the identity \eqref{eq:beta_I_identity} and after cancelling similar terms, we obtain 
\[
\begin{aligned} 
& h \sum_{k=0}^{N-1} \rho \sum_{i\in\mathbb{Z}} \Big[
\Big( - u_{i,k} + \int I[u_{k+1}] \left( y_i(\alpha) \right) + h L_{i,k}(\alpha) \, d\mu^\varepsilon_{i,k}(\alpha) \Big) \tilde{m}_{i,k} \\
&+ \Big( m_{i,k+1} - \sum_{n \in \mathbb{Z}} \int \beta_i(y_n(\alpha)) m_{n,k} \, d\mu^\varepsilon_{n,k}(\alpha) \Big) \tilde{u}_{i,k+1} \\
&+ h \Big( - \int \alpha m_{i,k} \, d\mu^\varepsilon_{i,k}(\alpha) \Big) \tilde{\varpi}_k \Big]
\geq 0.
\end{aligned}
\]
Now letting $\varepsilon \rightarrow 0$, we get:
\[
\Big\langle
A_{\rho,h}\begin{bmatrix}
w\end{bmatrix},
\begin{bmatrix}
\tilde{w}\end{bmatrix}
\Big\rangle \geq 0
\]
for arbitrary $\tilde{w}$, which yields 
that the second and third components of $A_{\rho,h} [w] $
must vanish everywhere, while the first component is zero 
at all indices such that $m_{i,k}>0$.
\end{proof}

Next, we show that any two solutions for the discretized operator $A_{\rho,h}$ have common minimizing points $\alpha^*$, and minimizing measures $\mu^*$.

\begin{proposition}\label{prop:uniqueness}
	Suppose Assumptions \ref{hyp:H_l_0},  \ref{hyp:V_uT_m0_2nd_derivative}  and \ref{hyp:compactness} hold.
	Let $ w = \{m, u, \varpi\}$ and  $ \tilde{w} = \{\tilde{m}, \tilde{u}, \tilde{\varpi}\}$  with $m, \tilde{m} > 0$ be two weak solutions for the operator $A_{\rho,h}$, defined by \eqref{eq:discrete_weak_solution}.
	
    Then $\tilde{\Lambda}_{i,k} \cap \Lambda_{i,k} \neq \varnothing$.
    Furthermore, if
    $A_{\rho,h}[w] = A_{\rho,h}[\tilde{w}] = 0$, for  $\mu^*_{i,k} \in \mathcal{M}[u]$ and $\tilde{\mu}^*_{i,k} \in \mathcal{M}[\tilde{u}]$, 
    we have $\tilde{\mu}^*_{i,k} \in \mathcal{M}[u]$ and $\mu^*_{i,k} \in \mathcal{M}[\tilde{u}]$.
\end{proposition}
\begin{proof}
From Lemma~\ref{lem:weak_strong_sol}, we have $A_{\rho,h}[w] = A_{\rho,h}[\tilde{w}] = 0$, for some selections of probability measures $\mu^*_{i,k} \in \mathcal{M}[u]$ and $\tilde{\mu}^*_{i,k} \in \mathcal{M}[\tilde{u}]$. Therefore, from the proof of Lemma~\ref{lem:monotonicity}, we get
\[
\begin{aligned}
 & h \sum_{k=0}^{N-1} \rho \sum_{i\in\mathbb{Z}} \Big[
 m_{i,k} \Big( \int I[\tilde{u}_{k+1}] \left( y_i(\alpha) \right) + h \tilde{L}_{i,k}(\alpha) d\mu^*_{i,k}(\alpha) \Big. \\
& \Big.  - \int I[\tilde{u}_{k+1}] \left( y_i(\alpha) \right) + h  \tilde{L}_{i,k}(\alpha) d\tilde{\mu}^*_{i,k}(\alpha) \Big) \\
& + \tilde{m}_{i,k} \Big( \int I[u_{k+1}] \left( y_i(\alpha) \right)  + h  L_{i,k}(\alpha) d\tilde{\mu}^*_{i,k}(\alpha) \Big. \\
& -  \Big. \int I[u_{k+1}] \left( y_i(\alpha) \right)  + h  L_{i,k}(\alpha) d\mu^*_{i,k}(\alpha) \Big) \Big] 
= 0.
\end{aligned}
\]
Since every term of the sum is nonnegative, we get
	\[
	\int I[u_{k+1}] \left( y_i(\alpha) \right)  + h  L_{i,k}(\alpha) d\tilde{\mu}^*_{i,k}(\alpha) 
	= \int I[u_{k+1}] \left( y_i(\alpha) \right)  + h  L_{i,k}(\alpha) d\mu^*_{i,k}(\alpha) .
	\]
    Therefore, there is a selection of $\tilde{\mu}^*_{i,k}$ supported on a subset of $\Lambda_{i,k}$. 
	Similarly, from the identity
    \[
    \int I[\tilde{u}_{k+1}] \left( y_i(\alpha) \right) + h \tilde{L}_{i,k}(\alpha) d\mu^*_{i,k}(\alpha) =
    \int I[\tilde{u}_{k+1}] \left( y_i(\alpha) \right) + h  \tilde{L}_{i,k}(\alpha) d\tilde{\mu}^*_{i,k}(\alpha),
    \]
    then, there is a selection of $\mu^*_{i,k}$ supported on a subset of $\tilde{\Lambda}_{i,k}$.
    Thus, $\tilde{\Lambda}_{i,k} \cap \Lambda_{i,k} \neq \varnothing$, and they have common minimizing measures.
\end{proof}

\subsection{Convergence}
Here, we prove that the limit of a solution of $A_{\rho,h}[w]=0$ is a weak solution to Problem~\ref{pro1}. 
To this end, we first show the consistency of the discretized operator $A^n:=A_{\rho_n,h_n}$.

\begin{lemma}\label{lem:consistency}
    Let Assumption \ref{hyp:H_l_0} and \ref{hyp:compactness} hold.
    Let $(\rho_n, h_n) \to 0$, and $\rho_n^2/h_n  \to 0$, and grid points $(x^n_i,t^n_k) \to (x,t)$, as $n \to \infty$.
    Take $\varphi = \{m, u, \varpi \} \in D$, with compactly supported $m \geq 0$, and  $\varphi^n = \{m^n, u^n, \varpi^n\}$ which take values $ \{m(x^n_i,t^n_k), u(x^n_i,t^n_k), \varpi(t^n_k) \}$ with $\mu^n_{i,k} \in \mathcal{M}[u^n]$.
    Then, 
    \[
		\lim_{n\to\infty}
		 \frac{1}{h_n}
        \Big\langle
		A^n\begin{bmatrix}
			\varphi^n
		\end{bmatrix}
		,
		\begin{bmatrix}
			\varphi^n
		\end{bmatrix}
		\Big\rangle
		=
		\Big\langle
		\mathcal{A}\begin{bmatrix}
			\varphi
		\end{bmatrix}
		,
		\begin{bmatrix}
			\varphi
		\end{bmatrix}
		\Big\rangle .
	\]	
\end{lemma}
\begin{proof}
For the Hamilton-Jacobi equation, we have shown the consistency in Lemma~\ref{lem:disc_prop}.
Now, we will show the consistency for the transport equation:
    \begin{align*}
         & \frac{1}{h_n} \sum_{i\in\mathbb{Z}}\rho_n \Big( 
        m^n_{i, k+1} u^n_{i, k+1} - \sum_{j\in\mathbb{Z}} \int \beta_{i}(x^n_{j} + h_n \alpha) m^n_{j, k}  \, d \mu^n_{j,k}(\alpha) \, u^n_{i, k+1} \Big)  \\
		= & \frac{1}{h_n} \sum_{i\in\mathbb{Z}} \rho_n \Big( 
  m^n_{i, k+1} u^n_{i, k+1} -   m^n_{i, k}  \int I[u^n_{k+1}](x^n_{i} + h_n \alpha)  \, d \mu^n_{i,k}(\alpha)  \Big) \\
  		= & \sum_{i\in\mathbb{Z}}\rho_n \Big( 
  \frac{m^n_{i, k+1} - m^n_{i, k}}{h_n} u^n_{i, k+1} - m^n_{i, k} \int \frac{I[u^n_{k+1}](x^n_i + h_n \alpha) - u^n_{i, k+1}}{h_n}  \, d \mu^n_{i,k}(\alpha) \Big),
    \end{align*}
    where on the second line, we used the identity \eqref{eq:beta_I_identity}. 
    Since $u$ is smooth, from estimate of interpolation function \eqref{eq:Inter_smooth}, we have
    \[
    I[u^n_{k+1}]\left( x^n_i + h_n \alpha \right) = u(x^n_i+h_n \alpha, t^n_{k+1}) + O\left(\rho^2_n \right).
    \]
    Then, the last line can be approximated as:
    \[ \sum_{i\in\mathbb{Z}}\rho_n \left( 
  m_t(x^n_i, t^n_k)  u^n_{i, k+1} - m^n_{i, k} u_x(x^n_i, t^n_{k+1}) \int \alpha \, d \mu^n_{i,k}(\alpha) + O(h_n) +  O\left( \frac{\rho^2_n}{h_n} \right) \right).
    \]
    Passing to the limit, when $n \to \infty$, the last expression converges to
    \[
    \begin{aligned}
    & \int_{\mathbb{R}} m_t(x,t) u(x,t) - \alpha^*(x,t) m(x,t) u_x(x,t) dx \\
    = & \int_{\mathbb{R}} (m_t(x,t) - (D_pH(x,\theta + u_x) m(x,t))_x ) \,u(x,t) dx,
     \end{aligned}
    \] 
    where $\alpha^*(x, t) = \lim_{n \to \infty} \int \alpha \, d\mu^n_{i,k}(\alpha)$, and we used the feedback form \eqref{eq:feedback}.

    The consistency of the discretized balance condition is straightforward.
\end{proof}

\begin{proposition}\label{prop:convergence}
	Suppose Assumptions \ref{hyp:H_l_0}-\ref{hyp:compactness} hold.
	Let $(\rho_n, h_n) \to 0$, and $\rho_n^2/h_n  \to 0$, and grid points $(x^n_i,t^n_k) \to (x,t)$, as $n \to \infty$.
    Let $w^n = \{m^n, u^n, \varpi^n\}$ be a weak solution of operator $A^n:=A_{\rho_n,h_n}$. 
	By $\tilde{w}^n = \{\tilde{m}^n, \tilde{u}^n, \tilde{\varpi}^n\}$ denote a step function that has values $\{m^n_{i,k}, u^n_{i,k}, \varpi^n_{k}\}$ on the domain $[i \rho_n,(i+1)\rho_n]\times[k h_n,(k+1)h_n]$.
	Then there exist subsequences $\tilde{m}^n \rightharpoonup m$ in $\mathcal{P}(\Omega_T)$,  $\tilde{u}^n \rightharpoonup^* u$ in $L^\infty(\Omega_T)$, over compact subsets, and $\tilde{\varpi}^n \rightharpoonup^* \varpi$ in $L^\infty([0,T])$, as $n \to \infty$.
	Moreover $w = \{m, u, \varpi\}$ is a weak solution to Problem~\ref{pro1}.
\end{proposition}

\begin{proof}
 From Lemma~\ref{lem:disc_estimates}, $\| \tilde{\varpi}^n \|_{L^\infty([0,T])}$ are uniformly bounded, $\|\tilde{u}^n\|_{L^\infty(K^n)}$ uniformly bounded over compact sets, $K^n \subset \Omega_T$, 
 and $\tilde{m}^n $ are probability measures on $\Omega_T$ by its definition, that is, $\|\tilde{m}^n(\cdot, t)\|_{L^1(\mathbb{R})}=1$, for all $t \in [0,T]$.
Furthermore, from Lemma~\ref{lem:compactness} all the functions $\tilde{m}^n$ are compactly supported on $[-R,R]$.
Therefore, by Prokhorov's theorem, there exists a subsequence, still denoted by $\tilde{m}^n$, such that $\tilde{m}^n \rightharpoonup m$ in $\mathcal{P}(\Omega_T)$.
Similarly, by Banach-Alaoglu theorem, there exist subsequences, still denoted by $\tilde{\varpi}^n$ and $\tilde{u}^n$, such that $\tilde{\varpi}^n \rightharpoonup^* \varpi$ in $L^\infty([0,T])$, and  $\tilde{u}^n \rightharpoonup^* u$ in $L^\infty(\Omega_T)$, over compact sets.

	Fix smooth functions $\varphi = \{\mu, v, \theta\} \in D_+^b$, such that $\mu \geq 0$ and $\int_\mathbb{R} \mu(x,t)dx=1$, and $\supp (\mu) \subseteq [-R,R]$.
	Define $\mu^n_{i,k} = \mu(i \rho_n, k h_n)$, $v^n_{i,k} = v(i \rho_n, k h_n)$ and $\theta^n_{k} = \theta(k h_n)$, and denote $\varphi^n = \{\mu^n, v^n, \theta^n\}$.
	Because $w^n$ is a weak solution to $A^n$, we have
	\[
	\begin{aligned}
		0 \leq 
		& \frac{1}{h_n}
        \Big\langle
		A^n\begin{bmatrix}
			\varphi^n
		\end{bmatrix}
		,
		\begin{bmatrix}
			\varphi^n
		\end{bmatrix}
		-
		\begin{bmatrix}
			w^n
		\end{bmatrix}
		\Big\rangle
		\\
		= & 
		\Big\langle
		\mathcal{A}\begin{bmatrix}
			\varphi
		\end{bmatrix}
		,
		\begin{bmatrix}
			\varphi
		\end{bmatrix}
		-
		\begin{bmatrix}
			\tilde{w}^n
		\end{bmatrix}
		\Big\rangle
        + O\left( h_n \right) + O\left( \frac{\rho_n^2}{h_n}\right).
	\end{aligned}
	\]	

	Then, using Lemma~\ref{lem:consistency} and passing to the limit, when $n \to \infty$ in this expression, yields
    \[
    \Big\langle
		\mathcal{A}\begin{bmatrix}
			\varphi
		\end{bmatrix}
		,
		\begin{bmatrix}
			\varphi
		\end{bmatrix}
		-
		\begin{bmatrix}
			w
		\end{bmatrix}
		\Big\rangle
        \geq 0
    \]
    where $w=\{ m, u, \varpi\}$.
\end{proof}

\begin{remark}
    In view of the Lipschitz estimates in space and time, Lemma~\ref{lem:disc_estimates} and Lemma~\ref{lem:time_lipschitz}, we can also extract a subsequence
$\tilde{u}^n$ that converges uniformly. 
If the limit is unique, in the convex case, for example, the whole sequence converges uniformly (see \cite{gomes2018mean} for a discussion).
\end{remark}

\section{Numerical results}\label{sec:numerical_results}

This numerical study aims to validate the accuracy and flexibility of our proposed Semi-Lagrangian scheme using a mean-field game with quadratic cost and a fluctuating supply function as a benchmark.
To evaluate the accuracy of our method, we use the semi-explicit solutions proposed in \cite{gomes2018mean}. 
Then, we compare our scheme with other numerical methods: the variational method using a potential transformation given in~\cite{ashrafyan2022variational} and a recurrent neural network method given in~\cite{ashrafyan2022potential}.

In our tests, we consider the terminal time $T=1$ and numerical space domain $[a,b]=[-1,1]$.
We divide $[a,b]$ spatial domain and $[0,T]$ time interval with $M$ and $N$ equal parts, respectively.
In our code, for interpolation basis functions we take $\mathbb{P}_1$-bases  $\{\beta_i\}_{i=1}^M$, defined by \eqref{eq:beta_i}. 

The minimization problem for the value function, \eqref{eq:HJ_discretization}, is solved using the \texttt{FindMin} native function in Mathematica software.
Suppose the point $y_i( \alpha_{i,k}^{*})$  is outside of the interval $[a,b]$; for example, it exits the interval from the right edge. In this case,  we update the time discretization parameter $h$, only for that node, such that $y_i( \alpha_{i,k}^{*}) = b$, and at the next time step, we restore the original $h$.

\subsection{Numerical Test 1}
Consider the quadratic cost
\[
l_0(\alpha) = \frac{c|\alpha|^2}{2}, \quad \mbox{and \;} V(x) =\frac{\eta}{2} (x - \tau)^2, 
\]
where $\tau\in \mathbb{R}$ and $\eta\geq 0$.
In this case, the Hamiltonian, $H$, is 
\[
H(x,\varpi+u_x)=\frac{\left|\varpi(t) + u_x(x,t)\right|^2}{2 c} - V(x),
\] 
leading to the following system 
\begin{equation}\label{eq:MFG_quadratic}
	\begin{cases}
		-u_t(x,t) + \frac{\left|\varpi(t) + u_x(x,t)\right|^2}{2 c}- V(x)=0, 
		\\
		m_t(x,t) - \big((\varpi(t) + u_x(x,t))m(x,t)\big)_x =0,
		\\
		-\int_{\mathbb{R}} (\varpi(t) + u_x(x,t))m(x,t) dx = Q(t),
	\end{cases}
\end{equation}
subject to the initial-terminal conditions \eqref{eq:initial_terminal}.

\subsubsection{Analytical solution}
For the system \eqref{eq:MFG_quadratic}, a semi-explicit solution is given in \cite{gomes2018mean}.
The solution $u(x,t)$ is a quadratic function in $x$ with time-dependent coefficients:
\[
u(t,x)=a_0(t) + a_1(t)x + a_2(t)x^2, \quad \text{for} \quad t\in[0,T],~ x \in \mathbb{R}.
\]
These coefficients are solutions to a Riccati-like system of ordinary differential equations.
The explicit price formula is given by
\[
\varpi(t)=\eta \left(\tau-\int_{\mathbb{R}} x \bar{m}(x) \, dx\right)(T-t)- \eta \int_t^T \int_0^s Q(r) \, dr \, ds - c Q(t),
\]
for $t\in[0,T]$, and the transport vector-field $b(t,x)$ is defined as
\[
b(t,x) = - \frac{1}{c} (\varpi(t) + a_1(t) + 2a_2(t)x), \quad \text{for} \quad t\in[0,T],~ x \in \mathbb{R},
\]
which is used to compute the distribution $m$ using the method of characteristics. 

\subsubsection{Fully discrete SL scheme}

Set $c = 1, \eta = 1$ and $\tau = 1/4$, then 
\[
V(x) =\frac{1}{2} \left(x - \frac{1}{4}\right)^2.
\]
Take $\bar{u}(x) \equiv 0$ and
\begin{equation}\label{eq:bar_m}
\bar{m}(x) = \frac{\hat{m}(x)}{\int_{-1}^1 \hat{m}(x) dx}, \, \mbox{where} \;
\hat{m}(x) = \begin{cases} 
	\exp\left(-\frac{1}{1-(1.1x)^2}\right), & |x|<1, \\
	0, & \mbox{otherwise}.
\end{cases}
\end{equation}
Finally, the supply function is the solution to the differential equation
\begin{equation}\label{eq:supply}
\begin{cases}
	\dot{Q}(t) = \overline{Q}(t) - \xi Q(t), & t \in [0,T], \\
	Q(0) = q_0,
\end{cases}
\end{equation}
where, as in \cite{ashrafyan2022variational}, we also choose  $\overline{Q}(t) = 5 \sin(3 \pi t)$, which denotes the average supply over time, $\xi = 4$ is the mean-reversion rate, and $q_0=-0.5$ is the initial supply level. 
These oscillations simulate market seasonality, allowing us to examine the impact of fluctuating supply on price variations. 

The problem data are set, and now we are ready to modify and solve our algorithm for the system \eqref{eq:MFG_quadratic}.

\textbf{The algorithm.}

\begin{enumerate}
	\item[0)]
	For $q=0$, take an initial guess for the price, $\varpi^q$ (in our code, we use $\varpi^0=-Q$), and let  $\varepsilon>0$ be the desired convergence tolerance for our algorithm.
	
	\item[1)]   For $k=N-1, N-2, \ldots, 0$ and  $i=0,1,\ldots,M$, compute the value function, $u$, using the discretization of the Hamilton-Jacobi equation:
	\begin{equation*}
		\begin{cases}
			\displaystyle u^q_{i,k} = \inf_{\alpha \in\mathbb{R}} \left\{  I[u^q_{k+1}](y_i(\alpha)) + h \, (l_0(\alpha)  + \varpi^q_k \alpha + V_i) \right\}, \\
			u^q_{i,N} = \bar{u}_i,
		\end{cases}
	\end{equation*}
    and find optimal controls:
	\[
	\alpha_{i,k}^{*q}= \arginf_{\alpha \in\mathbb{R}} \left\{ I[u^q_{k+1}](y_i(\alpha)) + h \, (l_0(\alpha)  + \varpi^q_k \alpha + V_i) \right\}.
	\]
	
	\item[2)] for $k=0,1,\ldots,N-1$ and $i=0,1,\ldots,M$,compute the distribution function, $m$, using the discretization of the transport equation:
	\begin{equation*}
		\begin{cases}
			\displaystyle m^q_{i,k+1}=\sum_{j=0}^M \beta_i \left(y_j( \alpha_{j,k}^{*}) \right) m^q_{j,k},\\
			m^q_{i,0} = \bar{m}_i,
		\end{cases}
	\end{equation*}
	
	\item[3)] for $k=0,1,\ldots,N-1$ and $p \to p+1$, update the price function, $\varpi^p$, using the discretization of the balance condition:
	\begin{equation*}
		\varpi^{q+1}_k = \varpi^{q}_k + \sum_{i=0}^M \alpha_{i,k}^{*q} \ m^q_{i,k} \ \Delta x - Q_k,
	\end{equation*}
	
	\item[4)] Check the convergence of the price using the
	stopping criterion: 
	\[
	\| \varpi^{q+1} - \varpi^q \|_{\infty} < \varepsilon,
	\]
	if satisfied stop, otherwise, set $\varpi^q := \varpi^{q+1} $, and repeat steps $1)-4)$.
\end{enumerate}

Note that the Hamiltonian is quadratic in our particular example, which makes the price function update rule explicit.

\subsubsection{Approximations and Comparisons}

The approximated solutions for $m, u$, and $\varpi$ with the discretization parameters $\rho = h = 0.005$, and the tolerance parameter $\varepsilon = 0.004$ are shown in Figures~\ref{fig:u_w_Q} and \ref{fig:m_m_evol}.

\begin{figure}[htp]
	\centering
	\begin{subfigure}[t]{0.25\textwidth}
		\vskip0cm
		\centering        
		\includegraphics[width=\textwidth]{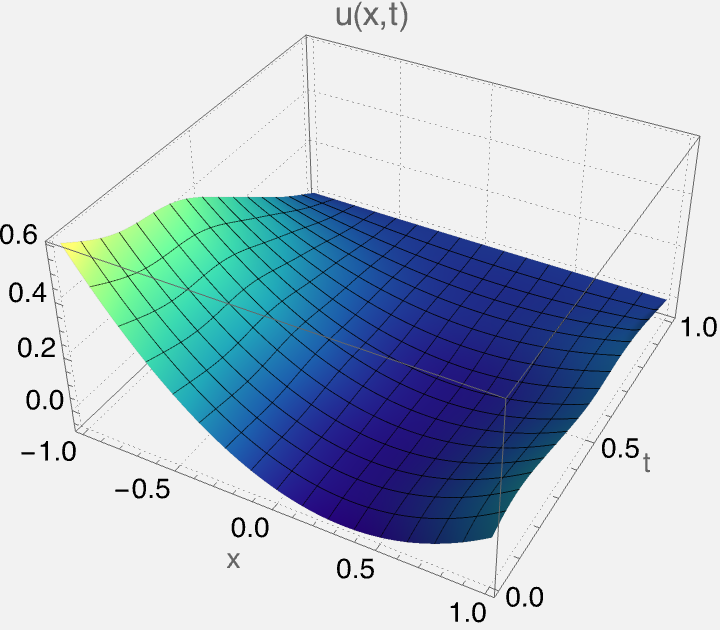}
	\end{subfigure}
 \hspace{15mm}
	\begin{subfigure}[t]{0.35\textwidth}
		\vskip0cm         
		\centering
		\includegraphics[width=\textwidth]{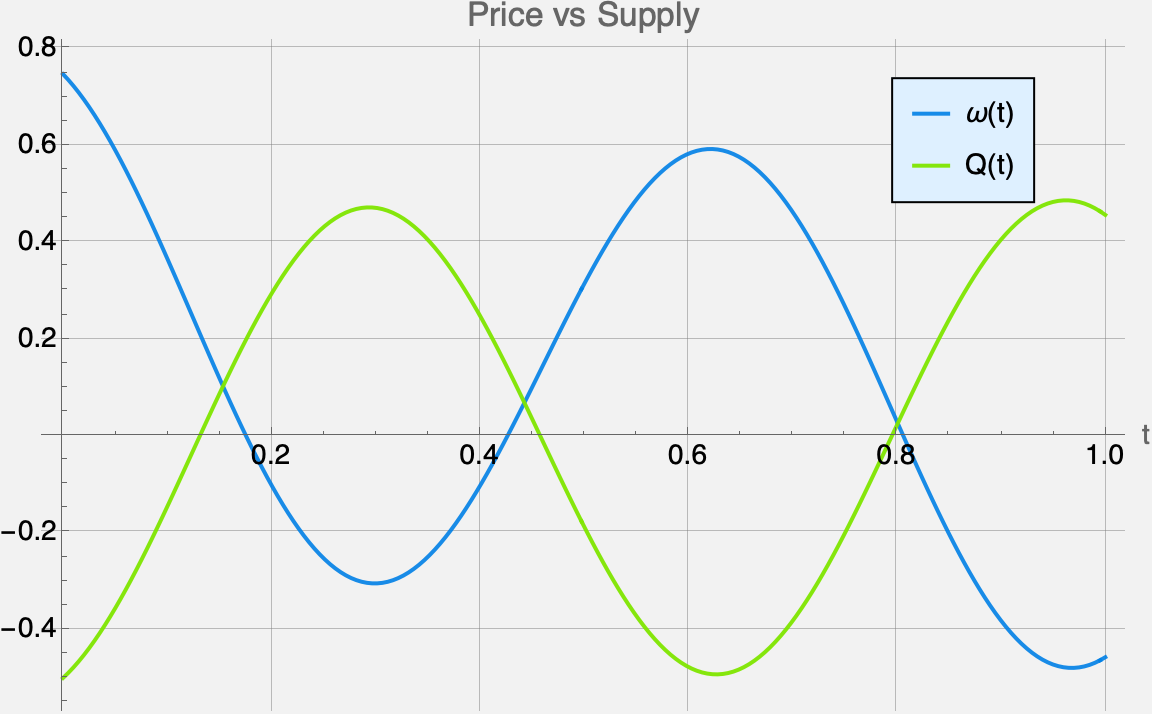}       
	\end{subfigure}
	\caption{Numerical solution of $u$ (left), and the relation of price $\varpi$ and supply functions, $Q$ (right).}
	\label{fig:u_w_Q}
\end{figure}

\begin{figure}[htp]
	\centering
	\begin{subfigure}[t]{0.25\textwidth}
		\vskip.2cm         
		\centering        
		\includegraphics[width=\textwidth]{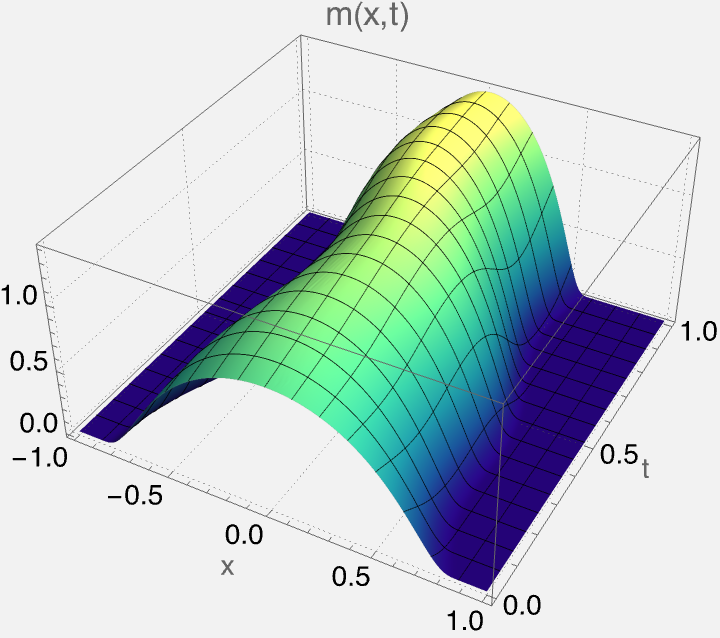}
	\end{subfigure}     
	\hspace{15mm}
	\begin{subfigure}[t]{0.25\textwidth}
		\vskip0cm
		\centering        
		\includegraphics[width=\textwidth]{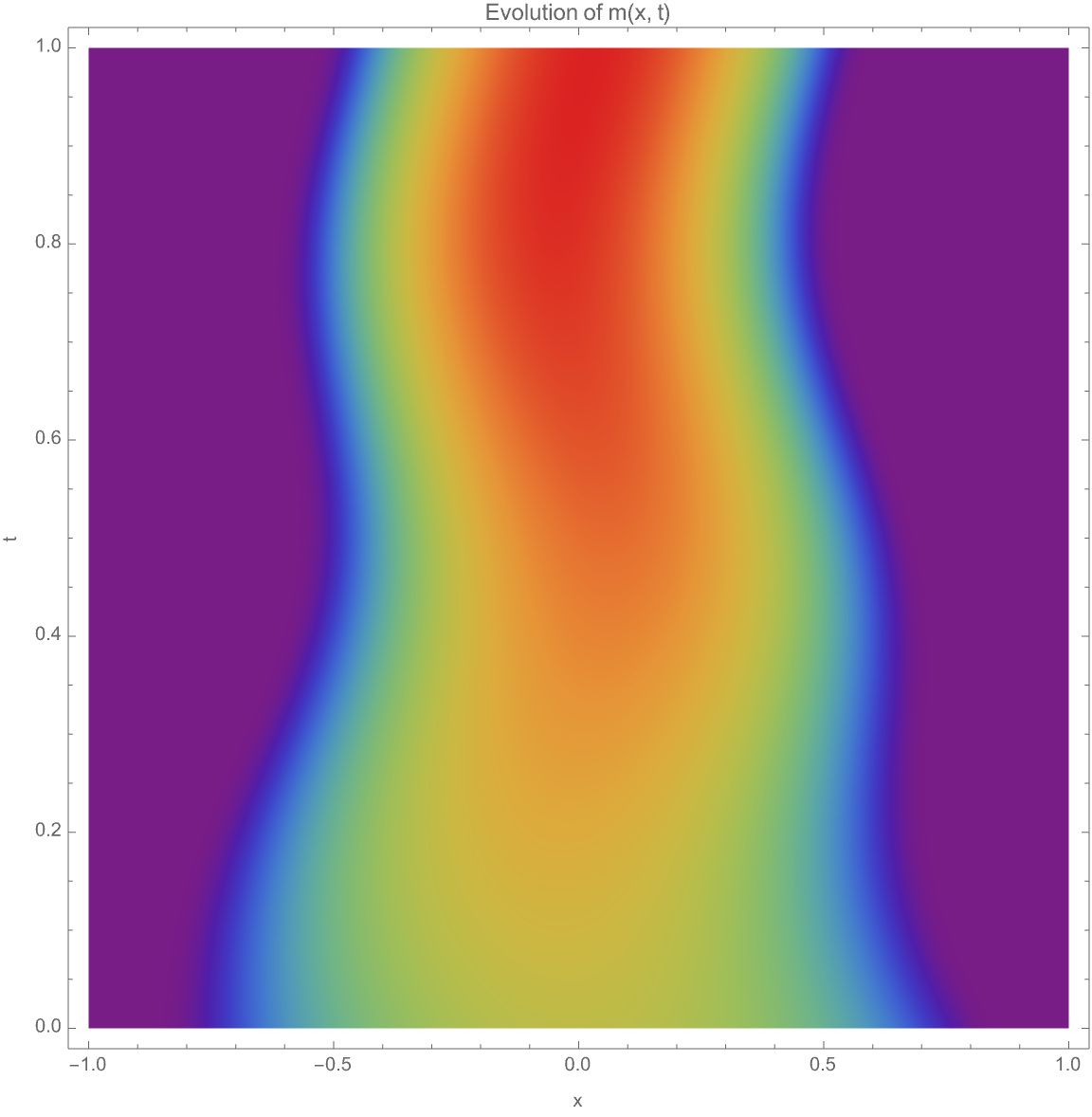}
	\end{subfigure}
	\caption{Numerical solution of $m$ (left) and its density plot (right).}
	\label{fig:m_m_evol}
\end{figure}

Figures~\ref{fig:werror}, \ref{fig:merror}, and \ref{fig:uerror} compare numerical solutions to analytical ones and illustrate their errors.

\begin{figure}[htp]
	\centering
	\vskip0cm         
	\centering        
	\includegraphics[width=.8\textwidth]{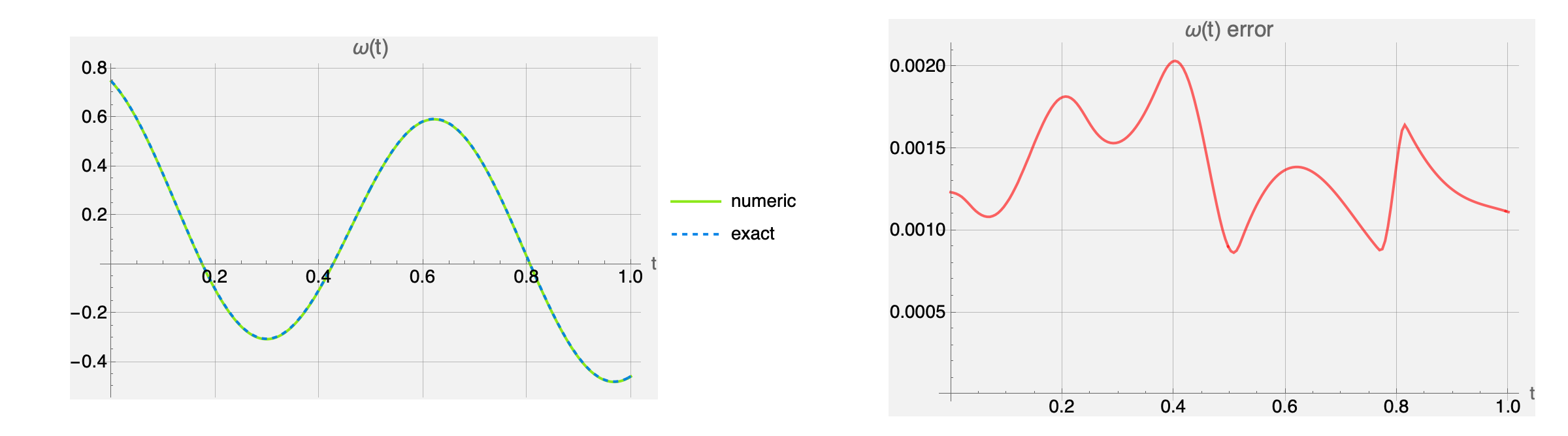}
	\caption{Approximated and exact solutions for $\varpi$ (left), and their absolute difference (right).}
	\label{fig:werror}
\end{figure} 

\begin{figure}[htp]
	\centering
	\vskip0cm         
	\centering        
	\includegraphics[width=.8\textwidth]{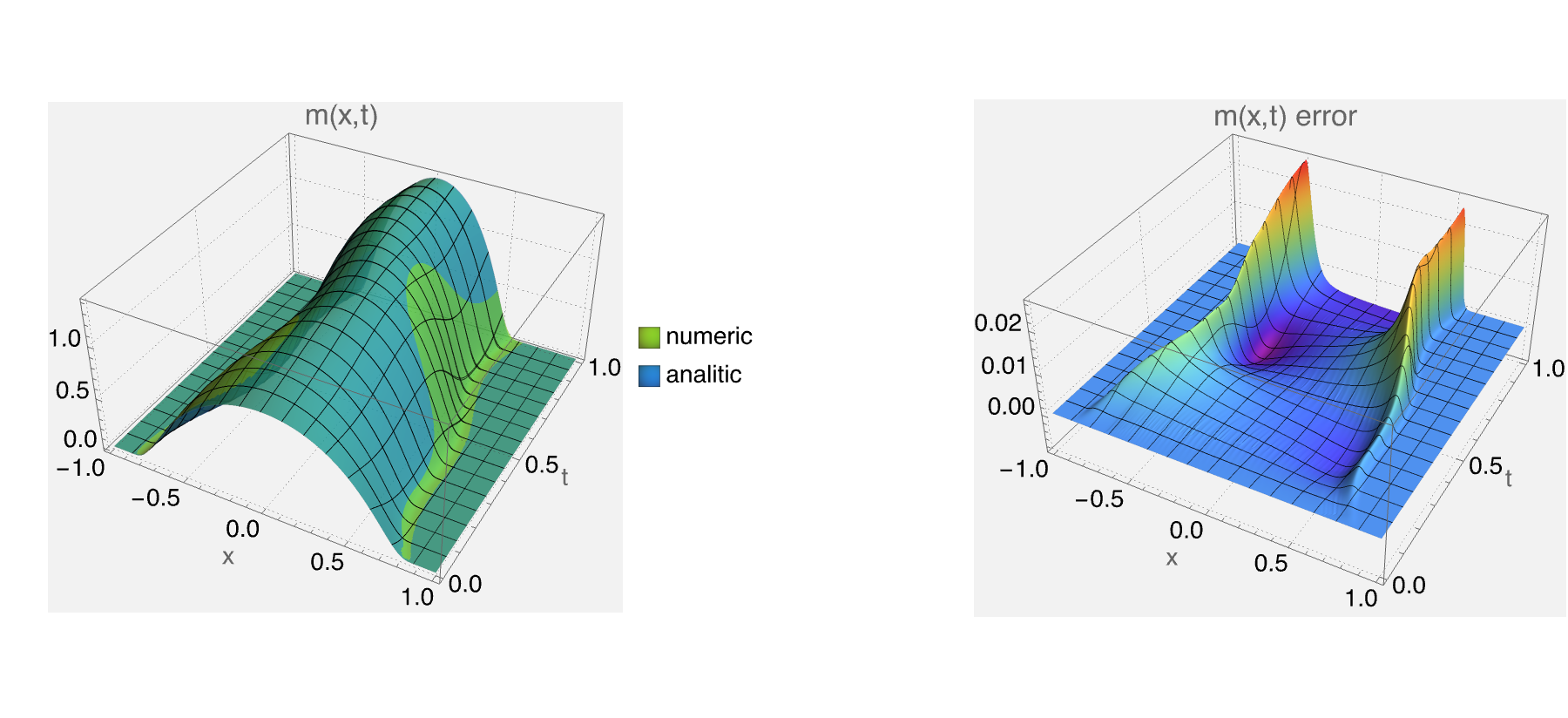}
	\caption{Approximated and exact solutions for $m$ (left), and their difference (right).  Note that they only have a major disagreement on the boundary of the support.}
	\label{fig:merror}
\end{figure} 

\begin{figure}[htp]
	\centering
	\vskip0cm         
	\centering        
	\includegraphics[width=.8\textwidth]{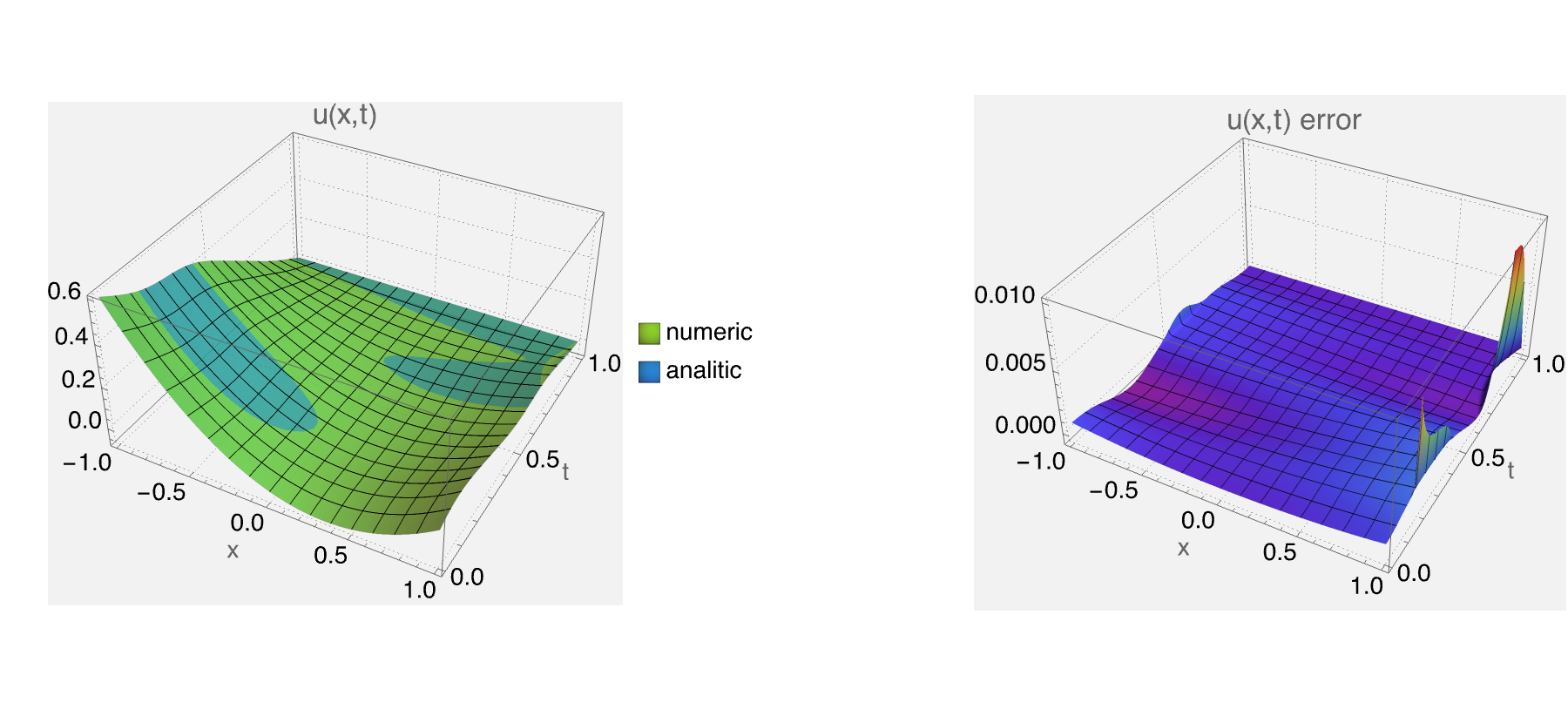}
	\caption{Approximated and exact solutions for $u$ (left), and their difference (right).}
	\label{fig:uerror}
\end{figure} 

Table~\ref{table:comparison} presents relative error rates for various grid sizes, with all cases converging within 4 iterations (one iteration is the calculation of steps $1)-4)$ of the algorithm), for the tolerance parameter $\varepsilon=0.004$. 
All the relative error rates are calculated using $l_\infty$ norms. For the value function $u$ we compute it at time $t=0$, and for the density function $m$ at $t=T$.

\begin{table}[ht]
	\centering
	\begin{tabular}{|c|c|c|c|c|}
		\hline
		$\rho$ & $h$  & err. $\varpi$ & err. $u$ & err $m$ \\ \hline
		$2.0 \cdot 10^{-2}$& $4.0 \cdot 10^{-2}$& $1.2 \cdot 10^{-3}$& $2.5 \cdot 10^{-2}$& $8.8 \cdot 10^{-2}$\\ \hline
		$1.0 \cdot 10^{-2}$ & $2.0 \cdot 10^{-2}$& $6.0 \cdot 10^{-3}$& $1.1 \cdot 10^{-2}$& $4.2 \cdot 10^{-2}$\\ \hline
  		$5.0 \cdot 10^{-3}$ & $1.0 \cdot 10^{-2}$& $3.3 \cdot 10^{-3}$& $4.7 \cdot 10^{-3}$& $1.8 \cdot 10^{-2}$\\ \hline
  		$2.5 \cdot 10^{-3}$ & $5.0 \cdot 10^{-3}$& $2.1 \cdot 10^{-3}$ & $1.2 \cdot 10^{-3}$& $6.7 \cdot 10^{-3}$\\ \hline
	\end{tabular}
	\caption{Comparing relative errors on different grids, for tolerance parameter $\varepsilon = 0.004$.}
	\label{table:comparison}
\end{table}

Figure~\ref{fig:total_error} visually illustrates how the total relative error decreases as $\rho, h \to 0$. 
The figure is based on the space nodes, with the relation of grid parameters $\rho = 2 h$, as in Table~\ref{table:comparison},

\begin{figure}[htp]
	\centering
	\vskip0cm         
	\centering        
	\includegraphics[width=.6\textwidth]{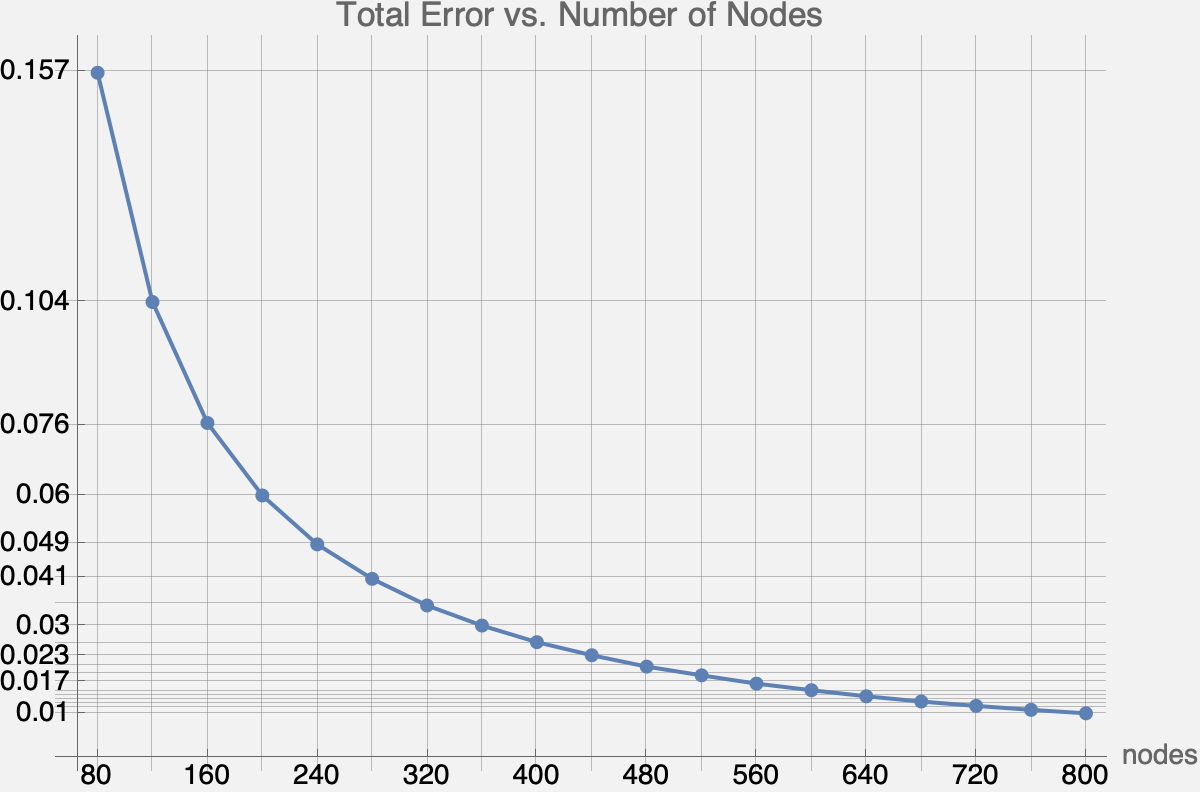}
	\caption{Visual representation of the total relative error as a function of the number of nodes as $\rho, h \to 0$, and a tolerance parameter set to $\varepsilon = 0.004$.}
	\label{fig:total_error}
\end{figure} 

We also run our scheme with the same discretization parameters $\rho=0.05$ and $h=0.05$ and the same problem data as in~\cite{ashrafyan2022potential}, where authors used the variational method.  
Our scheme outperforms the variational method in accuracy for $\varpi$, $u$, and $m$ by approximately 10, 8, and 1.5 times.
Additionally, our scheme reduces the computation time by a factor of 20 compared to the variational method.
It outperformed the variational method at every tested grid size, showing greater accuracy and faster convergence.

We also compared our scheme to the recurrent neural network method with the same problem data as given in~\cite{ashrafyan2022variational}. 
With the discretization parameters $\rho=0.01$ and $h=0.01$,
our scheme outperforms the recurrent neural network method in accuracy for $\varpi$, $u$, and $m$ by approximately
and 15, 10, and 3 times, respectively.
In terms of computational time, our scheme achieves the same level of accuracy in a matter of seconds, while the RNN requires several minutes to reach similar results.

We should stress, however, that the machine learning code was partially developed to address common noise problems that our method does not cover.

\subsection{Numerical Test 2}

Consider the following setup with a convex cost function, and a linear potential and terminal functions:
\[
l_0(\alpha) = \frac{3|\alpha|^{4/3}}{4}, \quad V(x) = \tau + \eta x, \quad \bar{u}(x) = \bar{a}_0 + \bar{a}_1 x.
\]
In this case, the Hamiltonian, \( H \), is a fourth-order function given by
\[
H(x,\varpi+u_x)=\frac{\left|\varpi(t) + u_x(x,t)\right|^4}{4} - V(x),
\]
which leads to the following system:
\[
\begin{cases}
& -u_t(x,t) + \frac{|u_x(x,t) + \varpi(t)|^4}{4} - V(x) = 0, \\
& m_t(x,t) - \left( (\varpi(t) + u_x(x,t))^3 m(x,t) \right)_x = 0, \\
& -\int_{\mathbb{R}} (\varpi(t) + u_x(x,t))^3 m(x,t) \, dx = Q(t),
\end{cases}
\]
subject to initial-terminal conditions \eqref{eq:initial_terminal}. 

\subsubsection{Analytical solution}
The solution, \( u(x,t) \), is a linear function in $x$ with time-dependent coefficients:
\[
u(x,t) = a_0(t) + a_1(t) x.
\] 
Moreover, from the balance condition, we gather
\[
 -(\varpi(t) + a_1(t))^3  = Q(t).
\]
Accordingly, we get
\[
\begin{cases}
& -a_0'(t) - a_1'(t) x + \frac{Q^{4/3(t)}}{4} = \tau + \eta x, \\
& a_0(T) + a_1(T) x = \bar{a}_0 + \bar{a}_1 x.
\end{cases}
\]

Using the terminal condition, $u(x, T) = \bar{a}_0 + \bar{a}_1 x$, we find explicit formulas for the price function, \( \varpi(t) \), and the value function, \( u(x,t) \), expressed as:
\[
\varpi(t) = -Q(t)^{1/3}  - (T - t) \eta - \bar{a}_1,
\]
and
\[
u(x,t) = - \int_t^T \frac{Q(s)^{4/3}}{4} ds + (T - t)V(x) + \bar{u}(x).
\]

Furthermore, 
we have
\[
\begin{cases}
& m_t(x,t) + Q(t) m_x(x,t) = 0, \\
& m(x,0) = \bar{m}(x), 
\end{cases}
\]
which can be solved explicitly:
\[
m(x,t) = \bar{m}\left( x - \int_0^t Q(s) ds \right).
\]

\subsubsection{Fully discrete SL scheme}

Set \( \eta = 1 \), \( \tau = 0 \), \( \bar{a}_0 = 0 \), and \( \bar{a}_1 = 0 \). 
Then $V(x) = x$, and  $\bar{u}(x) \equiv 0$. 
We take the  initial density function to be:
\[
\bar{m}(x) = \frac{\hat{m}(x)}{\int_{-1}^1 \hat{m}(x) \, dx}, \quad \text{with} \quad \hat{m}(x) = 
\begin{cases} 
\exp\left(-\dfrac{1}{1 - (1.2x)^2}\right), & |x| < 1, \\ 
0, & \text{otherwise}, 
\end{cases}
\]
and the same supply function, \( Q(t) \), as defined in  \eqref{eq:supply}.

For this setup, only step 3) of the semi-Lagrangian scheme requires modification due to the higher-order nature of the Hamiltonian. Specifically, since $D_p H(p) = p^3$, the price update rule is  given in an implicit form: 
\[
    \sum_{i=0}^N \left( \varpi^{q+1}_k - \varpi^q_k + (- \alpha_{i,k}^{*q})^{1/3} \right)^3 m^q_{i,k} \Delta x = -Q_k,
\]
for $k = 0,1, \ldots,N-1$.

\subsubsection{Approximations and Comparisons}

The approximated solutions for $m, u$, and $\varpi$ with the discretization parameters $\rho = h = 0.005$, and the tolerance parameter $\varepsilon = 0.0002$ are shown in Figure~\ref{fig:2m_u_w_Q}.

\begin{figure}[htp]
	\centering
    \begin{subfigure}[t]{0.25\textwidth}
		\vskip.0cm         
		\centering        
		\includegraphics[width=\textwidth]{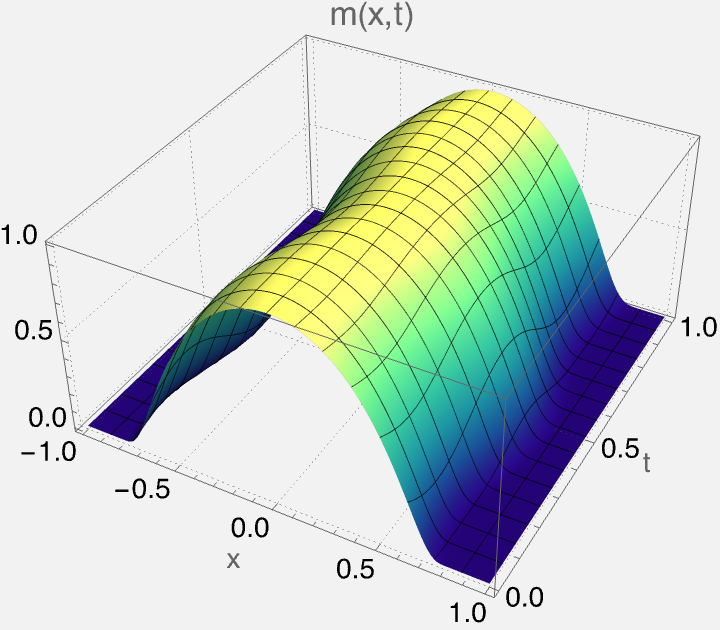}
	\end{subfigure}     
	\hspace{5mm}
	\begin{subfigure}[t]{0.25\textwidth}
		\vskip0cm
		\centering        
		\includegraphics[width=\textwidth]{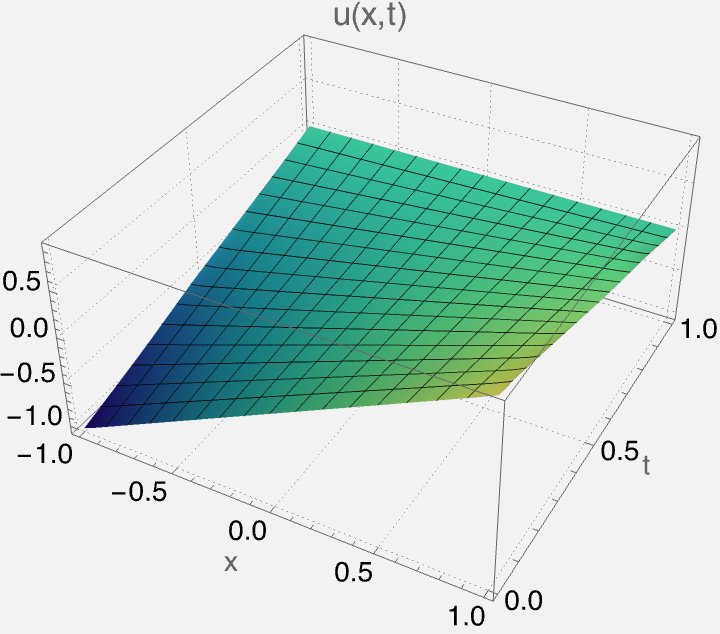}
	\end{subfigure}
 \hspace{5mm}
	\begin{subfigure}[t]{0.35\textwidth}
		\vskip0cm         
		\centering
		\includegraphics[width=\textwidth]{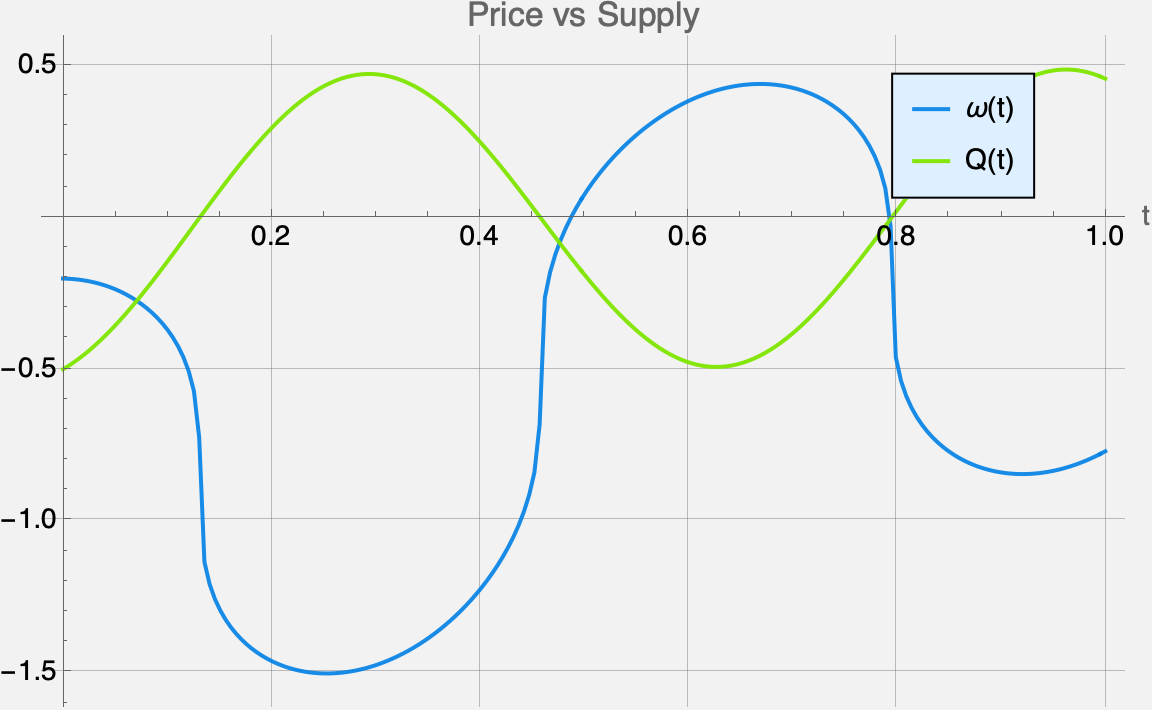}       
	\end{subfigure}
	\caption{Numerical solution of $m$, $u$, and the relation of price $\varpi$ and supply functions, $Q$.}
	\label{fig:2m_u_w_Q}
\end{figure}

Figures~\ref{fig:2werror}, \ref{fig:2merror}, and \ref{fig:2uerror} compare numerical solutions to analytical ones and illustrate their errors.

\begin{figure}[htp]
	\centering
	\vskip0cm         
	\centering        
	\includegraphics[width=.8\textwidth]{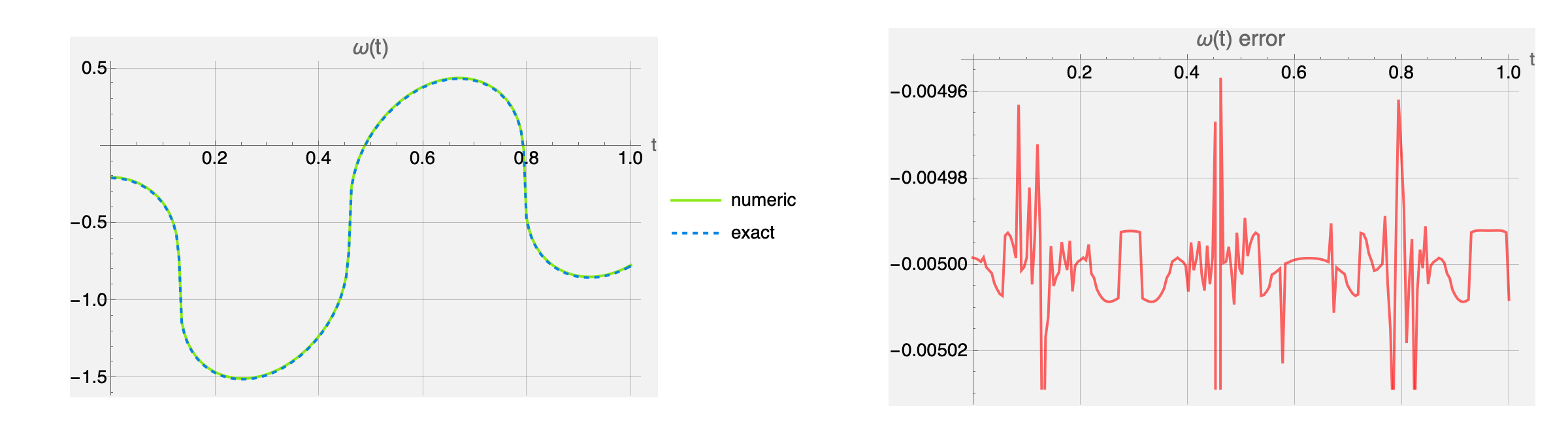}
	\caption{Approximated and exact solutions for $\varpi$ (left), and their absolute difference (right).}
	\label{fig:2werror}
\end{figure} 

\begin{figure}[htp]
\centering
	\begin{subfigure}[t]{0.25\textwidth}
		\vskip.0cm         
		\centering        
		\includegraphics[width=\textwidth]{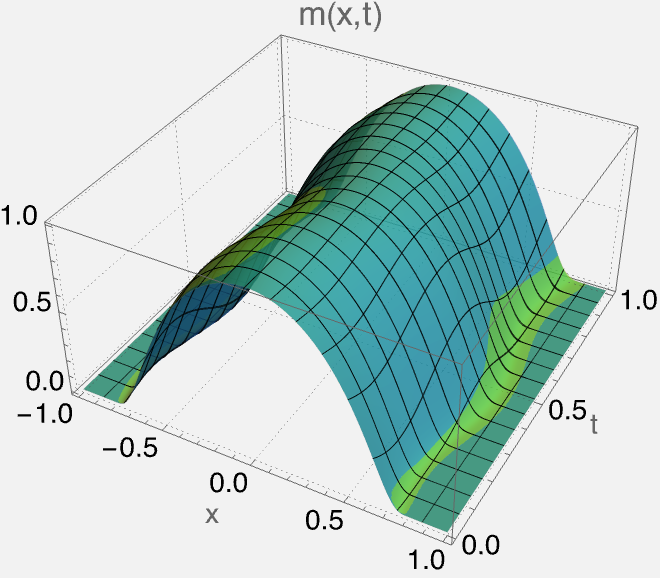}
	\end{subfigure}     
	\hspace{15mm}
	\begin{subfigure}[t]{0.27\textwidth}
		\vskip0cm
		\centering        
		\includegraphics[width=\textwidth]{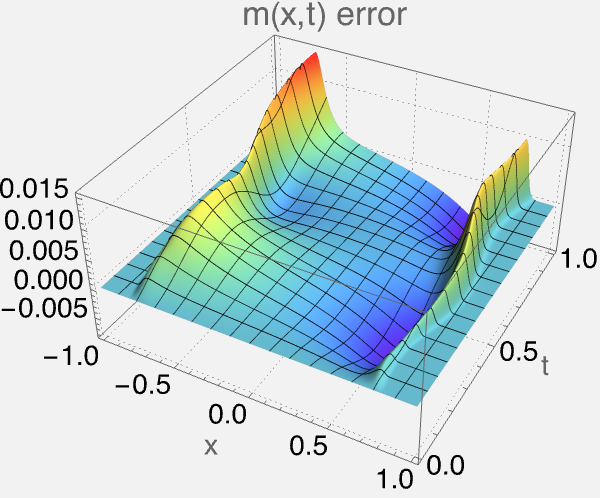}
	\end{subfigure}
	\caption{Approximated and exact solutions for $m$ (left), and their difference (right).  Note that they only have a major disagreement on the boundary of the support.}
	\label{fig:2merror}
\end{figure} 

\begin{figure}[htp]
	\centering
    \begin{subfigure}[t]{0.25\textwidth}
		\vskip.0cm         
		\centering        
		\includegraphics[width=\textwidth]{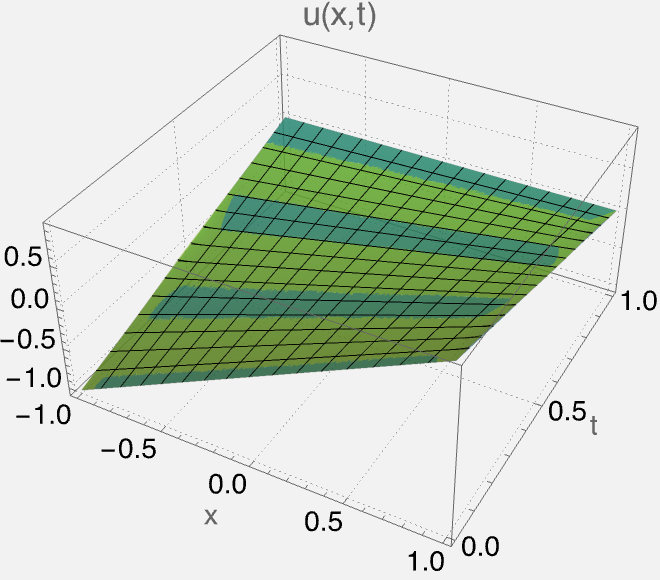}
	\end{subfigure}     
	\hspace{15mm}
    \begin{subfigure}[t]{0.27\textwidth}
	\vskip0cm         
     \centering   
	\includegraphics[width=\textwidth]{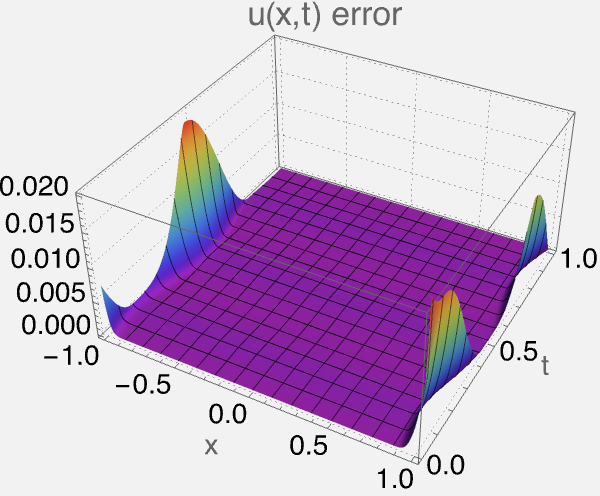}
    \end{subfigure} 
	\caption{Approximated and exact solutions for $u$ (left), and their difference (right).}
	\label{fig:2uerror}
\end{figure} 

Table~\ref{table:2comparison} presents relative error rates for various grid sizes, with all cases converging within 3 iterations (one iteration is the calculation of steps $1)-4)$ of the algorithm), for the tolerance parameter $\varepsilon=0.0002$. 
All the relative error rates are calculated using $l_\infty$ norms. For the value function $u$ we compute it at time $t=0$, and for the density function $m$ at $t=T$.

\begin{table}[ht]
	\centering
	\begin{tabular}{|c|c|c|c|c|}
		\hline
		$\rho$ & $h$  & err. $\varpi$ & err. $u$ & err $m$ \\ \hline
		$2.0 \cdot 10^{-2}$& $4.0 \cdot 10^{-2}$& $2.6 \cdot 10^{-2}$& $8.5 \cdot 10^{-3}$& $5.3 \cdot 10^{-2}$\\ \hline
		$1.0 \cdot 10^{-2}$ & $2.0 \cdot 10^{-2}$& $1.3 \cdot 10^{-2}$& $7.4 \cdot 10^{-3}$& $2.6 \cdot 10^{-2}$\\ \hline
  		$5.0 \cdot 10^{-3}$ & $1.0 \cdot 10^{-2}$& $6.7 \cdot 10^{-3}$& $6.8 \cdot 10^{-3}$& $1.3 \cdot 10^{-2}$\\ \hline
  		$2.5 \cdot 10^{-3}$ & $5.0 \cdot 10^{-3}$& $3.9 \cdot 10^{-3}$ & $6.5 \cdot 10^{-3}$& $6.6 \cdot 10^{-3}$\\ \hline
	\end{tabular}
	\caption{Comparing relative errors on different grids, for tolerance parameter $\varepsilon = 0.0002$.}
	\label{table:2comparison}
\end{table}

Figure~\ref{fig:2total_error} visually illustrates how the total relative error decreases as $\rho, h \to 0$. 
The figure is based on the space nodes, with the relation of grid parameters $\rho = 2 h$.

\begin{figure}[htp]
	\centering
	\vskip0cm         
	\centering        
	\includegraphics[width=.6\textwidth]{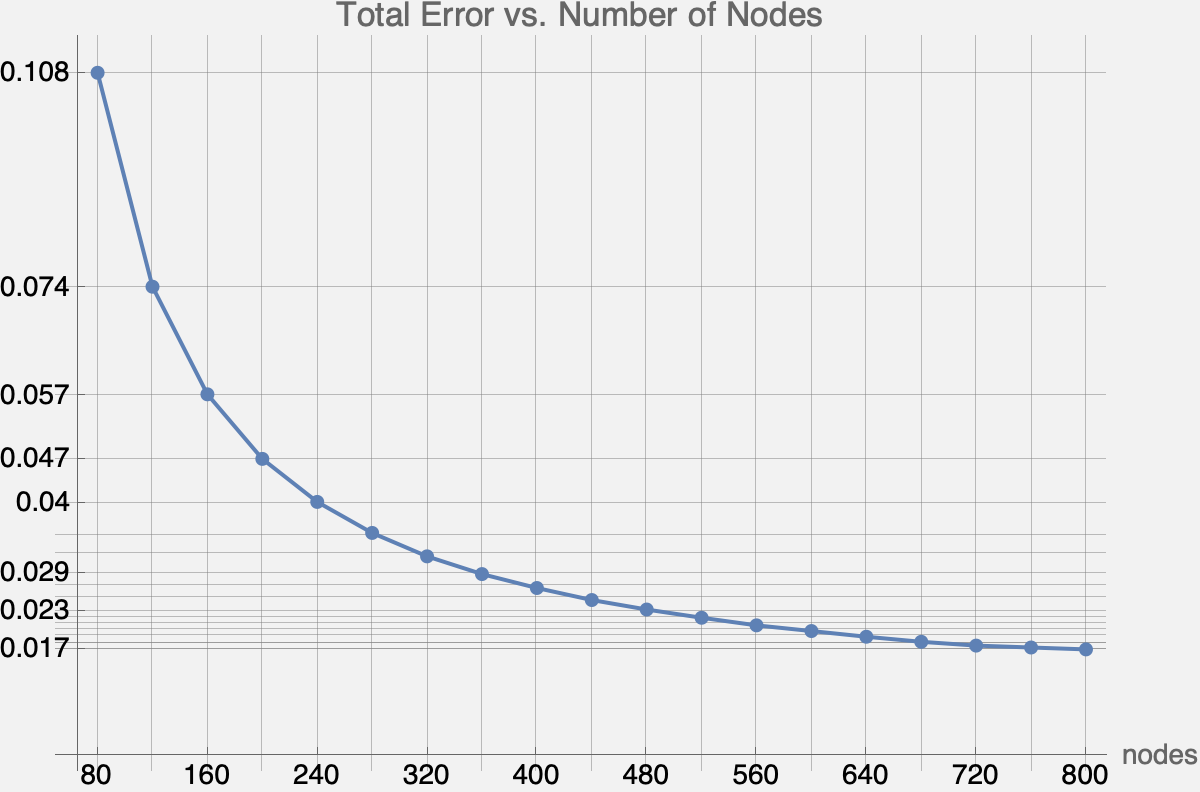}
	\caption{Visual representation of the total relative error as a function of the number of nodes as $\rho, h \to 0$, and a tolerance parameter set to $\varepsilon = 0.002$.}
	\label{fig:2total_error}
\end{figure}

For comparison, we also implemented the variational method presented in \cite{ashrafyan2022potential} for the second numerical test.
Our semi-Lagrangian scheme consistently outperformed the variational approach, converging more rapidly across all tested grid sizes and achieving an accuracy of roughly an order of magnitude higher.
It is worth noting that the variational method failed to converge when the grid size exceeded one thousand points.
A possible explanation is that the variational method requires derivatives of the potential function as well as a large number of constraints that need to be satisfied simultaneously, complicating its computational efficiency.
Furthermore, it requires up to second-order derivatives to reconstruct the problem's solutions.

\section*{Declarations}

\paragraph{\bf Ethical Statements:} This work considers mathematical models. As such, no human or animal studies were conducted.

\paragraph{\bf Competing interests:} Not applicable.

\paragraph{\bf Authors' contributions :} Both authors contributed equaly for this publication.

\paragraph{\bf Funding: } The research reported in this paper was funded through King Abdullah University of Science and Technology (KAUST) baseline funds and KAUST OSR-CRG2021-4674.

\bibliographystyle{plain}
\bibliography{mfg.bib}

\end{document}